\pdfoutput=1
\documentclass[a4paper,12pt,reqno,oneside]{amsart}

\usepackage[utf8x]{inputenc}
\usepackage[T1]{fontenc}
\usepackage{geometry}
\usepackage{lmodern}
\usepackage[stretch=10]{microtype}
\usepackage{tikz-cd}

\usepackage{amsmath}
\usepackage{amssymb}
\usepackage{amsthm}
\usepackage[matrix,arrow,cmtip]{xy}
\usepackage{bbm}
\usepackage{stmaryrd}
\usepackage[shortlabels]{enumitem}
\PassOptionsToPackage{hyphens}{url}
\usepackage[bookmarksnumbered,bookmarksdepth=2,bookmarksopen,colorlinks,linkcolor=black,citecolor=black,urlcolor=black,linktoc=all]{hyperref}
\usepackage[capitalise]{cleveref}

\newcommand{\Qbar}{\overline \bQ}
\newcommand{\bA}{\mathbb{A}}
\newcommand{\bC}{\mathbb{C}}
\newcommand{\bG}{\mathbb{G}}
\newcommand{\bN}{\mathbb{N}}
\newcommand{\bP}{\mathbb{P}}
\newcommand{\bQ}{\mathbb{Q}}
\newcommand{\bZ}{\mathbb{Z}}
\newcommand{\AAA}{\mathbb{A}}
\newcommand{\CC}{\mathbb{C}}
\newcommand{\GG}{\mathbb{G}}
\newcommand{\PP}{\mathbb{P}}
\newcommand{\QQ}{\mathbb{Q}}
\newcommand{\RR}{\mathbb{R}}
\newcommand{\ZZ}{\mathbb{Z}}

\newcommand{\aE}{\tilde{\mathcal{E}}}

\newcommand{\cA}{\mathcal{A}}
\newcommand{\cE}{\mathcal{E}}
\newcommand{\cG}{\mathcal{G}}
\newcommand{\cH}{\mathcal{H}}
\newcommand{\cO}{\mathcal{O}}
\newcommand{\cP}{\mathcal{P}}
\newcommand{\cS}{\mathcal{S}}

\newcommand{\gGL}{\mathbf{GL}}
\newcommand{\gH}{\mathbf{H}}
\newcommand{\gM}{\mathbf{M}}
\newcommand{\gSL}{\mathbf{SL}}

\newcommand{\rM}{\mathrm{M}}

\DeclareMathOperator{\Aut}{Aut}
\DeclareMathOperator{\End}{End}
\DeclareMathOperator{\GL}{GL}
\DeclareMathOperator{\Hom}{Hom}
\DeclareMathOperator{\ord}{ord}
\DeclareMathOperator{\Spec}{Spec}
\DeclareMathOperator{\Tay}{Tay}

\newcommand{\id}{\mathrm{id}}

\newcommand{\iotaan}{{\iota\textrm{-}\mathrm{an}}}

\newcommand{\abs}[1]{\lvert #1 \rvert}
\newcommand{\powerseries}[2]{#1 [\![ #2 ]\!]}

\newcommand{\ov}{\overline}

\newcommand{\fullmatrix}[4]{\left( \begin{matrix} #1 & #2 \\ #3 & #4 \end{matrix} \right)}

\newcommand{\defterm}[1]{\textbf{#1}}

\newtheorem{lemma}{Lemma}[section]
\newtheorem{proposition}[lemma]{Proposition}
\newtheorem{theorem}[lemma]{Theorem}
\newtheorem{corollary}[lemma]{Corollary}
\newtheorem{conjecture}[lemma]{Conjecture}
\Crefname{conjecture}{Conjecture}{Conjectures} 

\theoremstyle{definition}
\newtheorem*{definition}{Definition}
\newtheorem{remark}[lemma]{Remark}

\usepackage{ifthen}
\newcounter{constant}

\newcommand{\newC}[1]{%
   \ifthenelse{\equal{#1}{*}} {%
      \stepcounter{constant} c_{\theconstant}%
   } {%
      \refstepcounter{constant} c_{\theconstant} \label{C:#1}%
   }%
}
\newcommand{\refC}[1]{c_{\ref*{C:#1}}}

\usepackage{color}

\makeatletter
\def\subsection{\@startsection{subsection}{2}%
\z@{.9\linespacing\@plus.5\linespacing}{-.5em}%
{\normalfont\bfseries}}
\makeatother

\title[Zilber--Pink in a product of modular curves]{Zilber--Pink in a product of modular curves assuming multiplicative degeneration}

\author{Christopher Daw}
\author{Martin Orr}

\address{Daw: Department of Mathematics and Statistics, University of Reading,
    White\-knights,  PO Box 220,  Reading,  Berkshire RG6 6AX,  United Kingdom}
\email{chris.daw@reading.ac.uk}

\address{Orr: Department of Mathematics, The University of Manchester, Alan Turing Building, Oxford Road, Manchester M13 9PL, United Kingdom}
\email{martin.orr@manchester.ac.uk}

\dedicatory{Dedicated to the memory of Susan Daw (1954--2022)}


\begin{document}

\begin{abstract}
We prove the Zilber--Pink conjecture for curves in $Y(1)^n$ whose Zariski closure in $(\mathbb{P}^1)^n$ passes through the point $(\infty, \ldots, \infty)$, going beyond the asymmetry condition of Habegger and Pila.
Our proof is based on a height bound following André's G-functions method.
The principal novelty is that we exploit relations between evaluations of G-functions at unboundedly many non-archimedean places.
\end{abstract}

\maketitle

\vspace{-1.5em}

\section{Introduction}

One of the central problems of unlikely intersections is the {\bf Zilber--Pink conjecture} for Shimura varieties:

\begin{conjecture}\label{ZP}
Let $V$ be an irreducible subvariety of a Shimura variety $S$.  If $V$ is not contained in a special subvariety of positive codimension, then the intersection of $V$ with the union of all special subvarieties of~$S$ of codimension greater than $\dim(V)$ is not Zariski dense in $V$.
\end{conjecture}

An immediate corollary of Conjecture~\ref{ZP} is the Andr\'e--Oort conjecture: the special subvarieties of $S$ are characterised among the irreducible subvarieties of $S$ as those containing a Zariski dense set of special points. A proof of the Andr\'e--Oort conjecture was recently announced by Pila, Shankar and Tsimerman~\cite{PST+}, who, using work of Esnault and Groechenig, give height bounds for special points, thereby completing a new strategy of Binyamini, Schmidt and Yafaev~\cite{BSY}.

In \cite{HP12}, Habegger and Pila made the first progress on Conjecture~\ref{ZP} beyond the Andr\'e--Oort conjecture, proving it for so-called ``asymmetric curves'' in $Y(1)^n$.
Here $Y(1)$ denotes the level-one modular curve, which is isomorphic to the affine line $\AAA^1$.
An irreducible curve $C\subset Y(1)^n$ is called \defterm{asymmetric} if the restrictions of the coordinate functions on $Y(1)^n$ to $C$ have distinct degrees, up to at most one exceptional pair of coordinates, and ignoring coordinates whose restrictions to~$C$ are constant.
Habegger and Pila's theorem was as follows.

\begin{theorem} \cite[Theorem~1]{HP12} \label{HPthm}
Let $C$ be an irreducible asymmetric curve in $Y(1)^n$ defined over $\Qbar$.
If $C$ is not contained in a special subvariety of positive codimension, then the intersection of $C$ with the union of all special subvarieties of $Y(1)^n$ of codimension at least~$2$ is finite.
\end{theorem}

Removing the asymmetry condition in \cref{HPthm} has been a major challenge.
The main result of this article is \textbf{a proof of \cref{ZP} for a different class of curves in $Y(1)^n$},
namely, those that ``intersect infinity.''
We say that an irreducible curve in~$Y(1)^n \cong \AAA^n$ \defterm{intersects infinity} if its Zariski closure in $(\bP^1)^n$ contains the point $(\infty,\ldots,\infty)$.

\begin{theorem} \label{mainZP}
Let $C$ be an irreducible curve in $Y(1)^n$ which intersects infinity and is defined over $\Qbar$.
If $C$ is not contained in a special subvariety of positive codimension, then the intersection of $C$ with the union of all special subvarieties of $Y(1)^n$ of codimension at least~$2$ is finite.
\end{theorem}

The class of curves which intersect infinity neither contains nor is contained in the class of asymmetric curves.
Examples of curves in $Y(1)^n$ which intersect infinity but are not asymmetric are given by lines in $\AAA^n$ which are not parallel to any coordinate hyperplane.

In fact, by combining \cref{mainZP} with the ideas about ``modular Mordell--Lang'' from \cite{HP12}, we may prove the Zilber--Pink conjecture in full for lines in~$Y(1)^n$.
(We thank an anonymous referee for suggesting this theorem.)

\begin{theorem} \label{ZP-lines}
Let $C$ be a line in $Y(1)^n \cong \AAA^n$ which is defined over $\Qbar$.
If $C$ is not contained in a special subvariety of positive codimension, then the intersection of $C$ with the union of all special subvarieties of $Y(1)^n$ of codimension at least~$2$ is finite.
\end{theorem}

If a line in $Y(1)^n$ is not parallel to any coordinate hyperplane and is contained in any hyperplane defined by setting two coordinates equal to one another, then it is not contained in any special subvariety of positive codimension.
Hence we obtain the following concrete example of a new case of the Zilber--Pink conjecture included in \cref{mainZP}, where $\Phi_M(X,Y)\in\ZZ[X,Y]$ denotes the classical modular polynomial.

\begin{corollary}\label{ZP-lines-ab}
Let $a,b$ be distinct, non-zero algebraic numbers.
Then there are only finitely many $z\in\bC$ for which there exist positive integers $M$, $N$ such that the following equation holds:
\[ \Phi_M(z, z+a) = \Phi_N(z, z+b) = 0. \]
\end{corollary}

\begin{proof}
    Let $C$ denote the line in $Y(1)^3$ defined by $x_2=x_1+a$ and $x_3=x_1+b$.
    Then $C$ intersects infinity and is not contained in a special subvariety of positive codimension.
    The curves in~$Y(1)^3$ defined by $\Phi_M(x_1,x_2) = \Phi_N(x_1,x_3) = 0$ are special subvarieties, so the theorem follows from \cref{mainZP} applied to~$C$.
\end{proof}

As well as proving \cref{mainZP}, the aim of this article is also to \textbf{demonstrate a new method of proving lower bounds for Galois orbits}, the most difficult aspect of the Zilber--Pink conjecture.
We use a method of André \cite[Ch.~X]{And89} for studying fibres with large endomorphism rings in a family of abelian varieties, based on G-functions.
Recently, this method has been used to obtain bounds for Galois orbits in other cases of the Zilber--Pink conjecture \cite{BM:AO}, \cite{ExCM}, \cite{QRTUI}.
André's method has also been applied by Papas in variations of Hodge structure, going beyond families of abelian varieties~\cite{Papas}.

The key new contribution of this article is to use relations between evaluations of G-functions \textit{at all non-archimedean places} to obtain lower bounds for Galois orbits, and thus new cases of the Zilber--Pink conjecture.
Previous applications used only archimedean evaluations of G-functions (except \cite{And95}, which exploited only a fixed finite number of non-archimedean places).
Ignoring non-archimedean places restricted the large endomorphism rings for which the method could prove Galois orbit bounds, preventing the method being applied to $Y(1)^n$, or to intersections with  $E^2$ curves in the moduli space of abelian surfaces in \cite{QRTUI}.

This article's method can be generalised to other cases of the Zilber--Pink conjecture for curves in moduli spaces of abelian varieties, subject to a condition generalising the ``intersects infinity'' condition in \cref{mainZP}.
This condition corresponds to multiplicative degeneration of the corresponding family of abelian varieties as in \cite[X, Theorem~1.3]{And89}.
The method can also be used to prove versions of \cref{mainZP} with some variations of this degeneration condition, but removing it entirely is challenging.

\subsection{Height bounds and bounds for Galois orbits} \label{subsec:bounds-intro}

For a precise description of special subvarieties in $Y(1)^n$, see \cite[Section~2.1]{HP12}.
Informally, a \defterm{special subvariety} is a subvariety of $Y(1)^n$ defined by a system of equations of two types:
\begin{enumerate}[(i)]
\item $\Phi_M(x_i, x_j) = 0$ for some indices~$i, j \in \{ 1, \dotsc, n \}$;
\item $x_i = a$ for some index~$i \in \{ 1, \dotsc, n \}$ and some CM $j$-invariant $a \in \Qbar$.
\end{enumerate}
A \defterm{strongly special subvariety} of $Y(1)^n$ is a subvariety defined by equations of type~(i) only.

In \cite[Theorem~2]{HP12}, the variant of \cref{HPthm} which dealt with intersections between $C$ and non-strongly special subvarieties was proved, without the asymmetry restriction.
Thus in order to prove \cref{mainZP}, it suffices to restrict our attention to strongly special subvarieties.

Furthermore, in \cite{HP12}, the asymmetry condition was used only in proving the lower bound for Galois orbits \cite[Lemma~4.2]{HP12}.
Thus, in this article, we only need to prove a new lower bound for Galois orbits of intersections between strongly special subvarieties and a curve which intersects infinity.  \Cref{mainZP} then follows via the Pila--Zannier method, exactly as in \cite{HP12}.
Our lower bound for Galois orbits is as follows.

\begin{theorem} \label{galois-bound}
Let $n\geq 3$ and let $C \subset Y(1)^n$ be an irreducible algebraic curve defined over $\ov\QQ$ that intersects infinity and is not contained in a special subvariety of positive codimension.
Let $\Sigma$ denote the set of points of~$C$ with at least two CM coordinates.
There exist positive constants $\newC{galois-bound-mult}$ and $\newC{galois-bound-exp}$ with the following property.
Let $i_1, i_2, i_3, i_4 \in \{ 1, \dotsc, n \}$ with $i_1 \neq i_2$, $i_3 \neq i_4$ and $\{i_1,i_2\} \neq \{i_3,i_4\}$.
For all points $s=(s_1,\dotsc,s_n) \in C(\ov\bQ)\setminus\Sigma$ satisfying
\[\Phi_M(s_{i_1},s_{i_2})=\Phi_N(s_{i_3},s_{i_4})=0\]
for some positive integers $M$ and~$N$,
we have 
\[ [\QQ(s):\QQ] \geq {\refC{galois-bound-mult}}\max\{M,N\}^{\refC{galois-bound-exp}}. \]
\end{theorem}

The statement of Theorem \ref{galois-bound} is the same as \cite[Lemma~4.2]{HP12}, except that the ``asymmetric curve'' condition has been replaced by ``intersects infinity'' (and we always impose the condition $\{i_1,i_2\}\neq\{i_3,i_4\}$, which is not really a change from \cite[Lemma~4.2]{HP12} because the condition of intersecting infinity implies that no coordinate function is constant on~$C$).

Habegger and Pila deduce \cite[Lemma~4.2]{HP12} from isogeny estimates of Masser--Wüstholz type and a height bound: after a suitable relabelling of the coordinates,
\begin{equation} \label{eqn:HP-height-bound}
h(s_{i_1}) \leq \newC* \max\{1,\log(M)\}.
\end{equation}
Here, and throughout the paper, $h:\AAA^n(\ov\bQ)\to\RR$ denotes the (logarithmic) height defined in \cite[Section~2.2]{HP12}.
This height bound \cite[(17)]{HP12}, which is essentially the same as \cite[Thm.~1.1]{Hab10}, is obtained by comparing heights of coordinates on~$C$ and Faltings heights, with a key input being the asymmetry condition.

Habegger conjectured that the height bound \eqref{eqn:HP-height-bound} should hold without requiring the asymmetry condition, and indeed that it should only require two coordinates related by one modular polynomial (as opposed to the two relations between three or four coordinates in \cref{galois-bound}):
\begin{conjecture} \label{habegger-conj} \cite[Conjecture]{Hab10}
Let $C \subset \AAA^2$ be an irreducible algebraic curve defined over~$\Qbar$ that is not special.
There exists a positive constant $\newC{habegger-conj-mult}$ such that, for all $s = (s_1,s_2) \in C(\Qbar)$ satisfying $\Phi_M(s_1, s_2) = 0$ for some positive integer~$M$, we have
\[ h(s) \leq \refC{habegger-conj-mult} \max \{ 1, \log(M) \}. \]
\end{conjecture}

The majority of this article is dedicated to proving a height bound, somewhat weaker than \cref{habegger-conj}, in the case of curves that intersect infinity. We shall prove the following height bound or, more precisely, a technical variant of it.

\begin{theorem} \label{andre-bound}
Let $C\subset\AAA^n$ and $\Sigma$ be as in \cref{galois-bound} (in particular, $C$ intersects infinity). For every $\epsilon>0$, there exist constants $\newC{andre-bound-mult}$ and $\newC{andre-bound-exp}$ such that, for all $s$, $M$, $N$ as in \cref{galois-bound},
we have
\[ h(s) < \refC{andre-bound-mult} \max\{ [\QQ(s):\QQ]^{\refC{andre-bound-exp}}, M^\epsilon \}. \]
\end{theorem}

The bound in \cref{andre-bound} is weaker than $\refC{habegger-conj-mult} \max \{ 1, \log(M) \}$ in \cref{habegger-conj}, but it is sufficient to deduce \cref{galois-bound} in the same way as \cite[Lemma~4.2]{HP12} follows from \eqref{eqn:HP-height-bound}.
Note that, unlike \cref{habegger-conj} and \cite[Thm.~1.1]{Hab10}, we require two modular relations in \cref{andre-bound}.
Two modular relations are required to ``eliminate $2\pi i$'' in the relations between archimedean periods (see \cref{arch-relation}).

We will not prove or use \cref{andre-bound} directly, but rather a variant applying to a suitable finite cover of~$C$.  \Cref{andre-bound} can be deduced from \cref{andre-bound2} in the same way as \cref{galois-bound} is deduced from \cref{galois-bound2} in Section~\ref{subsec:main-proofs}.

\subsection{Comparison with previous applications of André's method}

As in \cite[Ch.~X]{And89}, our proof of \cref{andre-bound} is based on studying the relative periods of a family of abelian varieties $\cA$ over the curve~$C$.
The Taylor series of some of these period functions (the ``locally invariant periods'') around
a point where $\cA$ degenerates to a torus (in the case of \cref{andre-bound}, this point is $(\infty,\dotsc,\infty)$) are power series with special Diophantine properties, called G-functions.

At points $s \in C$ where the corresponding abelian variety $\cA_s$ has additional endomorphisms, there are polynomial relations between the evaluations at~$s$ of these G-functions.
\Cref{andre-bound} follows using a theorem of Bombieri: if there is a polynomial relation between the evaluations of a set of G-functions at a $\Qbar$-point $s$, valid at all places for which $s$ lies inside the radius of convergence of the G-functions, and not coming from a ``trivial'' relation between the G-functions themselves, then the height of $s$ is bounded in terms of the degree of the polynomial relation.

In \cite{And89}, and in previous applications of \cite{And89} to Galois orbit bounds for the Zilber--Pink conjecture, only the archimedean evaluations of G-functions were interpreted as periods.
In order to apply Bombieri's theorem without finding relations between non-archimedean evaluations of G-functions, the extra endomorphisms of $\cA_s$ were used to ensure that $s$ lies outside the radius of convergence of the G-functions at all non-archimedean places.
This is the reason for the condition appearing in \cite{And89} that $\End(\cA_s)$ does not inject into $\End(\GG_m^g) \otimes_\ZZ \QQ \cong \rM_g(\QQ)$ (where $g = \dim(\cA_s)$).

On the other hand, in \cref{mainZP}, the exceptional abelian varieties $\cA_s$ have endomorphism rings which inject into $\rM_4(\QQ)$.
For example, a product of non-CM elliptic curves $E_1 \times E_2 \times E_3 \times E_4$, where there are isogenies $E_1 \to E_2$ and $E_3 \to E_4$, has endomorphism algebra $\rM_2(\QQ) \times \rM_2(\QQ)$.
Thus, in order to apply Bombieri's theorem, we need relations valid at non-archimedean places.

Relations between non-archimedean evaluations of G-functions were used in \cite[Théorème~1]{And95}.  However in this result, as in \cite{And89}, the global relation needed for Bombieri's theorem was obtained as a product of relations, one for each place.  Thus in order to control the degree of this global relation, it was necessary to limit the number of non-archimedean places at which the G-functions could converge (by imposing an integrality condition at other places).
Our construction of relations instead directly constructs a polynomial relation simultaneously valid at all non-archimedean places, based on the Tate uniformisation of elliptic curves in a family degenerating to a torus.
At archimedean places, we are still using the previous approach of taking a product of local relations, one for each place.

\enlargethispage{\baselineskip}

Since the first preprint of this article became public, several other papers have appeared, all building on the aforementioned approach to Zilber--Pink using relations between G-functions. These include the paper \cite{Urbanik} of Urbanik and the papers \cite{Papas,Papas24,PapasII} of Papas, as well as our sequel to this article \cite{DOAg}. For a detailed survey of the literature to date, see \cite[Section 1.D]{DOAg}.

\subsection{Effectivity}

Our height bound \cref{andre-bound} and Galois bound \cref{galois-bound} can be made effective.
The main work required to do this would be to determine the size and global radius of the G-functions $F_\lambda$ and $G_\lambda$ in equation~\eqref{eqn:FG-li} and the size and singularities of the associated differential operator.

As previously mentioned, the Zilber--Pink conjecture in the case of intersections between curves $C \subset Y(1)^n$ and \textit{non-strongly special} subvarieties was already proved by \cite[Theorem~2]{HP12}, without any condition on the curve~$C$.
When the special subvariety is defined by one CM coordinate and one pair of isogenous coordinates, the underlying Galois bound \cite[Lemma~4.4]{HP12} was ineffective due to its dependence on an estimate for Siegel zeros.
(The remaining case, namely the André--Oort conjecture for a curve in $Y(1)^n$, is known effectively by a different method, due to Kühne \cite{Kuh12}.)
Our proof of \cref{galois-bound} could be adapted to apply to intersections with non-strongly special subvarieties, provided that $C$ intersects infinity, thus leading to an effective Galois bound for such cases -- the bound on CM discriminants would come from a height bound as in \cref{andre-bound} and an estimate of Masser--Wüstholz type.

Note that the full strength of this paper's ideas are not required to deal with intersections with non-strongly special subvarieties in~$Y(1)^n$, because having one CM coordinate is sufficient to prevent such intersections being $p$-adically close to infinity for any~$p$.
Hence, a height bound can be proved for such intersections using only relations between archimedean periods, using the methods of \cite[Ch.~X]{And89}.

Thus, if the relevant point-counting step of the Pila--Zannier strategy were to be made effective, that could be combined with Galois bounds using the techniques of this article, to obtain an effective proof of the Zilber--Pink conjecture for curves in $Y(1)^n$ which intersect infinity.

\subsection{Structure of the paper}

We begin with some preliminary definitions and lemmas in Section~\ref{sec:prelim}. In Section~\ref{sec:families}, we describe two well-known families of elliptic curves: the ``$1/j$'' family and the Tate family.  We also describe periods of elliptic curves in terms of the relation between these families. In Section~\ref{sec:period-relations}, we construct polynomial relations between the periods of isogenous elliptic curves, both archimedean and non-archimedean.
In Section~\ref{sec:proofs}, we prove our main theorems: \cref{galois-bound} and a variant of \cref{andre-bound}, by interpreting the relations from Section~\ref{sec:period-relations} as relations between evaluations of G-functions.
We also deduce \cref{ZP-lines} from \cref{mainZP}.

\subsection*{Acknowledgements}
Both authors would like to thank the Engineering and Physical Sciences
Research Council for its support via New Investigator Awards [EP/S029613/1 to C.D., EP/T010134/1 to M.O.].
They would also like to thank the organisers of the Mathematisches Forschungsinstitut Oberwolfach workshop~2216, \textit{Diophantische Approximationen}, for the opportunity to speak about this work.
They are grateful to the referees for their careful reading
of the paper and helpful suggestions and corrections.

\section{Preliminaries}

\label{sec:prelim}

\subsection{Embeddings of a number field}

Let $K$ be a number field.
We define an \defterm{embedding} of $K$ to be a homomorphism $K \to \CC$ or $\CC_p$ for some prime~$p$.
If $\iota$ is an embedding of $K$, we write $\CC_\iota$ for the complete algebraically closed field which is the target of $\iota$. We write $D(0, r, \CC_\iota)$ for the open disc in~$\CC_\iota$ of radius~$r$, centred at~$0$.

If $s \in K$, we write $s^\iota = \iota(s) \in \CC_\iota$.
Similarly, if $S$ is an algebraic object over $K$ (in particular, a polynomial, a power series, an algebraic variety or a morphism of algebraic varieties), then we write $S^\iota$ for the base change of $S$ by $\iota \colon K \to \CC_\iota$.
This superscript notation sometimes carries the connotation that $S^\iota$ should be interpreted as an analytic, rather than algebraic, object over $\CC_\iota$: in particular, if $f \in \powerseries{K}{X}$ then $f^\iota$ denotes the analytic function obtained by evaluating $f$ on its disc of convergence in $\CC_\iota$.

We have sought to minimise the use of rigid geometry in this paper (it is essential only in section~\ref{sec:padic-rel}).
Nevertheless, whenever we wish to refer specifically to the complex or rigid analytification of an algebraic variety (after base change to~$\CC_\iota$), we denote this by the superscript $\iotaan$.

We normalise the absolute value on $\CC_\iota$ so that it extends the standard archimedean or $p$-adic absolute value on $\QQ$.

\subsection{Non-archimedean power series}

Let $f(X) = \sum_{n=0}^\infty a_nX^n \in \powerseries{K}{X}$ be a power series, where $(K, \abs{\cdot})$ is a complete field equipped with an absolute value.
We write $R(f)$ for the radius of convergence of $f$.
Recall that
\[ R(f) = \frac{1}{\limsup_{n \to \infty} \abs{a_n}^{1/n}}. \]
When $(K, \abs{\cdot})$ is non-archimedean, we also define
\[ R^\dag(f) = 1/\sup_{n \in \ZZ_{>0}} \abs{a_n}^{1/n} = \inf \{ \abs{\rho} : \rho \in K^\times, f(\rho^{-1}X)-f(0) \in \cO_\powerseries{K}{X} \}.  \]
Note that we ignore $a_0$ in the definition of $R^\dag$.

Clearly $R^\dag(f) \leq R(f)$.
It is easy to prove that if
$R(f) > 0$, then $R^\dag(f) > 0$.

The quantity $R^\dag(f)$ is useful for understanding composition of power series over non-archimedean fields.
We note the following bound, which can be seen as a consequence of the non-archimedean mean value theorem:

\begin{lemma} \label{padic-Rdag-bound}
Let $(K, \abs{\cdot})$ be a non-archimedean complete field equipped with an absolute value.
Let $g \in \powerseries{K}{X}$ and $x \in K$.
If $\abs{x} \leq R^\dag(g)$, then \[ \abs{g(x)-g(0)} \leq R^\dag(g)^{-1}\abs{x}. \]
\end{lemma}

\begin{proof}
Let $g(X) = \sum_{n=0}^\infty b_nX^n$.
Note that $R^\dag(g) = \inf_{n \geq 1} \abs{b_n}^{-1/n}$.
Hence, for each $n \geq 1$, $R^\dag(g)^n \leq \abs{b_n}^{-1}$.
Therefore, for each $n \geq 1$,
\begin{equation} \label{eqn:bnxn}
\abs{b_nx^n} \leq \abs{b_n}\abs{x}R^\dag(g)^{n-1}
\leq R^\dag(g)^{-1}\abs{x} \cdot \abs{b_n}R^\dag(g)^n
\leq R^\dag(g)^{-1}\abs{x} \cdot 1.
\end{equation}
Because the absolute value is non-archimedean, this implies that $\abs{g(x)-g(0)} \leq R^\dag(g)^{-1}\abs{x}$.
\end{proof}

The following lemma is obvious from the perspective of rigid geometry (it is a statement about composing morphisms of rigid open discs).
However we have given a proof in order to minimise this paper's dependence on rigid geometry.

\begin{lemma} \label{padic-composition}
Let $(K, \abs{\cdot})$ be a non-archimedean complete field equipped with an absolute value.
Let $f, g \in \powerseries{K}{X}$ with $R(f) \geq 1$ and $g(0) = 0$.
Then $R(f \circ g) \geq R^\dag(g)$ and, for all $x \in K$ satisfying $\abs{x} < R^\dag(g)$, we have $f(g(x)) = (f \circ g)(x)$.
\end{lemma}

\begin{proof}
Let $f(X) = \sum_{n=0}^\infty a_nX^n$ and $g(X) = \sum_{n=1}^\infty b_nX^n$.

Let $x \in K$ with $\abs{x} < R^\dag(g)$.
Let $M_{\abs{x}}(g) = \sup_{n \geq 1} \abs{b_nx^n}$.
By inequality \eqref{eqn:bnxn} from the proof of \cref{padic-Rdag-bound} and by our hypotheses, we have
\[ M_{\abs{x}}(g) \leq R^\dag(g)^{-1}\abs{x} < 1 \leq R(f). \]
Since also $\abs{x} < R(g)$, the conditions of \cite[6.1.5, Theorem]{Rob00} are satisfied.  Therefore $R(f \circ g) > \abs{x}$ and $f(g(x)) = (f \circ g)(x)$.
\end{proof}

Over $\CC$, composition of power series is described by the following classical result:
\begin{lemma} \label{arch-composition}
Let $f, g \in \powerseries{\CC}{X}$ with $g(0) = 0$.
Let
\[ s = \sup\{ r \leq R(g) : \abs{g(x)} < R(f) \text{ for all } x \in \CC \text{ with } \abs{x} < r \}. \]
Note that if $R(f) > 0$ and $R(g) > 0$, then $s > 0$.
Then $R(f \circ g) \geq s$ and, for all $x \in \CC$ satisfying $\abs{x} < s$, we have $f(g(x)) = (f \circ g)(x)$.
\end{lemma}

\subsection{G-functions}

\begin{definition} \cite[p.~1]{And89}
A \defterm{G-function} is a power series $y(X) = \sum_{n \geq 0} a_nX^n$ whose coefficients $a_n$ belong to a number field $K$,  which satisfies the following conditions:
\begin{enumerate}
\item there exists $\newC{G-function-arch-base}>0$ such that $\abs{a_n} < \refC{G-function-arch-base}^n$ for all $n$ and for all archimedean absolute values $\abs{\cdot}$ of~$K$;
\item there exists a sequence of positive integers $(d_n)$ which grows at most geometrically such that $d_na_m$ is an algebraic integer for all $m \leq n$;
\item $y(X)$ satisfies a linear homogeneous differential equation
\[ \frac{d^\mu}{dX^\mu}y + \gamma_{\mu-1} \frac{d^{\mu-1}}{dX^{\mu-1}}y + \dotsb + \gamma_1 \frac{d}{dX}y + \gamma_0y = 0 \]
with coefficients $\gamma_i \in K(X)$.
\end{enumerate}
\end{definition}


\begin{definition}
Let $y_1, \dotsc, y_n \in \powerseries{\Qbar}{X}$.
Let $\tilde Q \in \Qbar[X][Y_1, \dotsc, Y_n]$ be a homogeneous polynomial.
We say that $\tilde Q$ is a \defterm{functional relation} between $y_1, \dotsc, y_n$ if $\tilde Q(X)(y_1(X), \dotsc, y_n(X)) = 0$ in $\powerseries{\Qbar}{X}$.
\end{definition}

\begin{definition}
Let $y_1, \dotsc, y_n \in \powerseries{\Qbar}{X}$ and let $\xi \in \Qbar$.
Let $Q \in \Qbar[Y_1, \dotsc, Y_n]$ be a homogeneous polynomial.
Let $K$ be a number field which contains~$\xi$ and all the coefficients of $y_1, \dotsc, y_n$ and~$Q$.
We say that:
\begin{enumerate}
\item $Q$ is an \defterm{$\iota$-adic relation} between the evaluations at $\xi$ of $y_1, \dotsc, y_n$ (for an embedding $\iota$ of~$K$) if $\abs{\xi^\iota} < \min \{ 1, R(y_1^\iota), \dotsc, R(y_n^\iota) \}$ and
\[ Q^\iota(y_1^\iota(\xi^\iota), \dotsc, y_n^\iota(\xi^\iota)) = 0. \]
\item $Q$ is a \defterm{global relation} between the evaluations at $\xi$ of $y_1, \dotsc, y_n$ if it is an $\iota$-adic relation between the evaluations at~$\xi$ for every embedding $\iota$ of~$K$ satisfying $\abs{\xi^\iota} < \min \{ 1, R(y_1^\iota), \dotsc, R(y_n^\iota) \}$. Of course, the set of such embeddings may be empty.
\item $Q$ is a \defterm{trivial relation} between $y_1, \dotsc, y_n$ at~$\xi$ if it is the specialisation at $X=\xi$ of a functional relation $\tilde Q \in \Qbar[X][Y_1, \dotsc, Y_n]$ between $y_1, \dotsc, y_n$, where $\tilde Q$ is homogeneous of the same degree as~$Q$.
\end{enumerate}
\end{definition}

The following theorem of Andr\'e is a slight strengthening of a theorem due to Bombieri \cite[Thm.~4]{Bom81}. 

\begin{theorem} \cite[VII, Thm.~5.2]{And89} \label{hasse-principle}
Let $y_1, \dotsc, y_n \in \powerseries{\Qbar}{X}$ be G-functions which satisfy a linear system of differential equations
\begin{equation} \label{eqn:diff}
\frac{dy_i}{dX} = \sum_{j=1}^n a_{ij} y_j
\end{equation}
for some $a_{ij} \in \Qbar(X)$.
Then there exist constants $\newC{hasse-multiplier}$, $\newC{hasse-exponent}$, depending only on $y_1, \dotsc, y_n$, such that, for all $\xi \in \ov\QQ \setminus \{0\}$ and all $\delta \in \ZZ_{>0}$, if $\xi$ is not a singularity of the differential equation~\eqref{eqn:diff} and there exists a non-trivial global relation of degree~$\delta$ between the evaluations of $y_1, \dotsc, y_n$ at~$\xi$, then
\[ h(\xi) \leq \refC{hasse-multiplier} \delta^{\refC{hasse-exponent}}. \]
\end{theorem}

\begin{remark}\label{rem:are-G-fn}
Under the hypothesis that $y_1,\dotsc,y_n \in \powerseries{\Qbar}{X}$ satisfy a system of differential equations~\eqref{eqn:diff}, the condition that they are G-functions is equivalent to the condition $\sigma(y_1,\dotsc,y_n) < \infty$ in \cite[VII, Thm.~5.2]{And89} (see \cite[I, 1.3 and~1.4]{And89}).
\end{remark}

\begin{remark}
    The definitions of global and non-trivial relations in \cite[VII, 5.2]{And89} refer only to polynomials with coefficients in some number field~$K$.
However, nothing in the proof of \cite[VII, Thm.~5.2]{And89}, including the constants which appear in \cite[VII, Prop.~3.5]{And89}, depends on the number field~$K$, so \cref{hasse-principle} is valid for relations with coefficients in $\Qbar$.
\end{remark}

\begin{remark}
    It is not clear whether those $\xi$ for which $\abs{\xi^\iota} > \min \{ 1, R(y_1^\iota), \dotsc, R(y_n^\iota) \}$ for all embeddings $\iota$ of $K$ are formally covered by \cite[VII, 5.2]{And89}. For this case, we refer to \cite[Section 12]{Papas}.
\end{remark}

We can deduce the following theorem from \cref{hasse-principle}, which is the form in which we shall actually apply it.
Since the deduction of \cref{hasse-principle-E} from \cref{hasse-principle} is only an exercise in \cite{And89}, and we are not aware of a full proof of \cref{hasse-principle-E} in the literature, we have included a proof of this deduction.

\begin{theorem}\label{hasse-principle-E} \cite[Introduction, Thm.~E]{And89}
    Let $y_1,\ldots,y_n\in\powerseries{\Qbar}{X}$ be $G$-functions. There exist constants $\newC{hasse-multiplier-E}$, $\newC{hasse-exponent-E}$ depending only on $y_1, \dotsc, y_n$, such that, for all $\xi \in \ov\QQ \setminus \{0\}$ and all $\delta \in \ZZ_{>0}$, if there exists a non-trivial global relation of degree~$\delta$ between the evaluations of $y_1, \dotsc, y_n$ at~$\xi$, then
\[ h(\xi) \leq \refC{hasse-multiplier-E} \delta^{\refC{hasse-exponent-E}}. \]
\end{theorem}

\begin{proof}
    By definition, each $y_i$ satisfies a linear homogeneous differential equation of some order $d_i$. We label the power series 
    \[
    \frac{d^\mu}{dX^\mu}y_i\text{ for }i=1,\ldots,n\text{ and }\mu=0,1,\ldots,d_i
    \] 
    as $z_1, \dotsc, z_m$ in some order, such that $z_i=y_i$ for $1 \leq i \leq n$.
    Then there exists $\Gamma\in\rM_m(\Qbar(X))$ such that
    \[\frac{d}{dX}z_i=\sum_{j=1}^m\Gamma_{ij}z_j.\]

    Let $W$ denote the finite set given by \cref{trivial-relations-extra-variables}, applied to $z_1,\dotsc,z_m$ and $n$.
    For $\xi \in \Qbar \setminus W$, the theorem follows from \cref{hasse-principle} applied to $z_1, \dotsc, z_m$.
    Since $W$ is a finite set which depends only on $z_1, \dotsc, z_m$, we can modify the constants so that the theorem holds for all $\xi \in \Qbar \setminus \{0\}$.
\end{proof}

\begin{lemma} \label{trivial-relations-extra-variables}
    Let $y_1,\ldots,y_m\in\powerseries{\Qbar}{X}$. There exists a finite set $W\subset\Qbar$ such that, for all $n\leq m$ and all $\xi \in \Qbar \setminus W$, a non-trivial relation between $y_1,\ldots,y_n$ at $\xi\in\Qbar$ is a non-trivial relation between $y_1,\ldots,y_m$ at $\xi$.
\end{lemma}

\begin{proof}
    Let $I$ denote the ideal of $\Qbar[X][Y_1,\ldots,Y_m]$ generated by the (homogeneous) functional relations between $y_1,\ldots,y_m$. In other words, $I$ consists of the $f\in\Qbar[X][Y_1,\ldots,Y_m]$ satisfying: each homogeneous component of $f$ is in the kernel of the ring homomorphism
    \[\Qbar[X][Y_1,\ldots,Y_m]\to\powerseries{\Qbar}{X},\ Y_i\mapsto y_i.\]
    (Here ``homogeneous'' means homogeneous as a polynomial in $Y_1,\ldots,Y_m$ with coefficients in $\Qbar[X]$.)
    For $n\leq m$, the ideal of $\Qbar[X][Y_1,\ldots,Y_n]$ generated by functional relations between $y_1,\ldots,y_n$ is $I_n:=I\cap \Qbar[X][Y_1,\ldots,Y_n]$. 
    
    For any polynomial ring $R$ over  $\Qbar[X]$ and any $\xi\in\Qbar$, let ${\rm sp}_\xi$ denote the specialisation map on $R$ induced by $X\mapsto \xi$.
    By definition, the ideal of $\Qbar[Y_1,\dotsc,Y_m]$ generated by trivial relations between $y_1,\dotsc, y_m$ at~$\xi$ is ${\rm sp}_\xi(I)$, while the ideal  of $\Qbar[Y_1,\dotsc,Y_n]$ generated by trivial relations between $y_1,\dotsc, y_n$ at~$\xi$ is ${\rm sp}_\xi(I_n)$.
    Thus, in order to prove the lemma, it suffices to show that ${\rm sp}_{\xi}(I)_n:={\rm sp}_{\xi}(I)\cap\Qbar[Y_1,\ldots,Y_n]$ is contained in ${\rm sp}_{\xi}(I_n)$ for all $\xi\in\Qbar$ outside of a finite set.

    Consider the lexicographic monomial order on $\Qbar[X,Y_1,\ldots,Y_m]$ with $Y_m>Y_{m-1}>\cdots>Y_1>X$ and let $B$ denote a Gr\"obner basis for $I$ with respect to this ordering. 
    By \cite[Chapter 6, \S3, Proposition 1(ii), p308]{CLO15}, if we write each $b\in B$ as 
    \begin{align}\label{eqn:binB}
b=h_b(X)\mathbf{Y}^{\alpha_b}+(\text{terms}<\mathbf{Y}^{\alpha_b})
    \end{align}
    and let $W$ denote the union of the roots of the $h_b(X)$, then, for all $\xi\in\Qbar\setminus W$, the set ${\rm sp}_\xi(B)$ is a Gr\"obner basis for ${\rm sp}_\xi(I)$. Therefore, by \cite[Chapter 3, \S1, Theorem 2, p122]{CLO15}, ${\rm sp}_\xi(B)_n:={\rm sp}_\xi(B)\cap\Qbar[Y_1,\ldots,Y_n]$ is a Gr\"obner basis for ${\rm sp}_\xi(I)_n$. 

For all $\xi \in \Qbar \setminus W$ and $b \in B$, we have $h_b(\xi) \neq 0$.  Hence, by \eqref{eqn:binB} and our choice of monomial order, ${\rm sp}_{\xi}(b)\in\Qbar[Y_1,\ldots,Y_n]$ if and only if $b\in\Qbar[X][Y_1,\ldots,Y_n]$.
    That is, ${\rm sp}_\xi(B)_n={\rm sp}_\xi(B_n)$ for all $\xi\notin W$. We conclude that ${\rm sp}_\xi(I)_n$ is contained in ${\rm sp}_\xi(I_n)$ for all $\xi\notin W$. This concludes the proof. 
\end{proof}

\begin{remark}
The exceptional set~$W$ in \cref{trivial-relations-extra-variables} depends on the enlarged set of power series $y_1,\dotsc,y_m$, not just the smaller set $y_1,\dotsc,y_n$.
Indeed, given \emph{any} $\xi \in \Qbar \setminus \{0\}$ and any homogeneous polynomial $Q \in \Qbar[Y_1,\dotsc,Y_n]$, it is always possible to choose an additional power series $y_{n+1}$ such that $Q$ becomes trivial as a relation between $y_1,\dotsc,y_n,y_{n+1}$ at~$\xi$.
To achieve this, assume that $Q(y_1,\dotsc,y_n) \neq 0$ (otherwise, $Q$ is already a trivial relation between $y_1,\dotsc,y_n$, and it remains so after adding any additional power series).
Let $\ell = \ord_{X=0} Q(y_1,\dotsc,y_n)$ and let $d = \deg(Q)$.
Let
\[ z = \frac{Q(y_1,\dotsc,y_n)}{X^\ell(X-\xi)} \in \powerseries{\Qbar}{X}. \]
Then $z(0) \neq 0$, so $z$ has a $d$-th root in $\powerseries{\Qbar}{X}$, which we call~$y_{n+1}$.
Now
\[ \tilde Q(X)(Y_1,\dotsc,Y_{n+1}) = Q(Y_1,\dotsc,Y_n) - X^\ell(X-\xi)Y_{n+1}^d \]
is a homogeneous functional relation between $y_1,\dotsc,y_{n+1}$, which specialises to~$Q$ at~$\xi$.

This implies that our proof of \cref{hasse-principle-E} is not effective, if we assume given as input only information about the functional relations between $y_1,\dotsc,y_n$.

The proof of \cref{trivial-relations-extra-variables} is effective (in the sense that it gives an algorithm to compute~$W$) if we assume given as input a set of generators for the functional relations between the enlarged set of power series $y_1,\dotsc,y_m$.
It follows that our proof of \cref{hasse-principle-E} is effective if given as input generators for the functional relations among the power series denoted $z_1,\dotsc,z_m$ in the proof of \cref{hasse-principle-E} (that is, $y_1,\dotsc,y_n$ \emph{and certain derivatives}), as well as the information about singularities, size and radius of convergence needed to make \cref{hasse-principle} effective.
\end{remark}

\section{Families of elliptic curves} \label{sec:families}

\subsection{The ``\texorpdfstring{$1/j$}{1/j}'' family of elliptic curves} \label{subsec:j-family}

Let $S = \bA_\bQ^1 \setminus \{1/1728\}$ and let $S^* = S \setminus \{0\}$.
Let
\begin{multline*}
\cE = \bigl\{ ([x:y:z], s) \in \bP^2 \times S :
\\ y^2z + xyz = x^3 + \frac{36s}{1728s-1}xz^2 + \frac{s}{1728s-1}z^3 \bigr\} \setminus \{([0:0:1], 0)\}.
\end{multline*}
This is the $S$-smooth locus of an elliptic scheme $\pi \colon \cE \to S$.
Indeed, the fibres of $\cE$ over $S^*$ are elliptic curves and $\cE_s$ has $j$-invariant $1/s$.
The fibre~$\cE_0$ is the smooth locus of a nodal cubic.
Therefore, the standard Weierstrass group law makes $\cE$ into an $S$-group scheme, with $\cE_0 \cong \bG_m$.

We obtain an invariant holomorphic differential form on $\cE/S$ by setting
\[ \omega = \frac{dx}{2y+x} \in \Omega^1(\cE/S). \]

\subsection{The Tate family}

Let $\iota \colon \QQ \to \CC_\iota$ denote an embedding of $\QQ$.
Thus $\CC_\iota = \CC$ or $\CC_p$.
Let $\Delta_\iota$ denote the open unit disc $D(0, 1, \CC_\iota)$ and let $\Delta_\iota^* = \Delta_\iota \setminus \{0\}$.

We shall use \cite[Chapter~V]{SilATAEC} and \cite{TateCurves} as our main references for Tate curves.
Following these references, we will work mostly with explicit power series.  We will require rigid geometry only in equation~\eqref{eqn:phi-omega} and section~\ref{sec:padic-rel}, where it is needed to talk about the pullback of differential forms by rigid morphisms.
For an outline of the theory of differential forms on rigid spaces, see \cite[sec.~1]{BLR95}.

Note that \cite[Chapter~V]{SilATAEC} is stated in terms of finite extensions of $\QQ_p$ or~$\ov\QQ_p$, rather than over~$\CC_p$.
However, as noted in \cite[Remark~V.3.1.2]{SilATAEC}, the key theorem \cite[Theorem~V.3.1]{SilATAEC} remains valid over any non-archimedean complete field, in particular over $\CC_p$.

The $\iota$-adic Tate family of elliptic curves is defined, as in \cite[Theorems V.1.1 and~V.3.1]{SilATAEC} and \cite[(13)]{TateCurves}, by
\begin{multline} \label{eqn:delta-family}
\aE_\iota = \bigl\{ ([x:y:z], q) \in \PP^2(\CC_\iota) \times \Delta_\iota :
\\ y^2z + xyz = x^3 + \tilde{a}_4^\iota(q)xz^2 + \tilde{a}_6^\iota(q)z^3 \bigr\} \setminus \{([0:0:1], 0)\},
\end{multline}
where
\begin{align*}
    \tilde{a}_4(q) & = -5s_3(q) = -5q - 45q^2 - 140q^3 - \dotsb \in \powerseries{\ZZ}{q},
\\ \tilde{a}_6(q) & = -\frac{5s_3(q) + 7s_5(q)}{12} = -q - 23q^2 - 154q^3 - \dotsb \in \powerseries{\ZZ}{q},
\\  s_k(q) & = \sum_{n=1}^\infty \sigma_k(n)q^n \in \powerseries{\ZZ}{q},
\\ \sigma_k(n) & = \sum_{d|n, d > 0} d^k.
\end{align*}
For each $q \in \Delta^*_\iota$, the fibre $\aE_{\iota,q}$ is an elliptic curve over $\CC_\iota$.
The $j$-invariant of $\aE_{\iota,q}$ is given by the $\iota$-adic evaluation of the $q$-expansion of the standard $j$-function:
\[ j(q) = q^{-1} + 744 + 196884q + \dotsb \in \powerseries{\QQ}{q}[1/q]. \]
Similarly to Section~\ref{subsec:j-family}, the fibre $\aE_{\iota,0}$ is the smooth locus of a nodal cubic, so the Weierstrass group law makes $\aE_{\iota,0}$ into an algebraic group over~$\CC_\iota$, isomorphic to~$\bG_m$.

Again by \cite[Theorems V.1.1, V.3.1]{SilATAEC} or by \cite[Theorem~1]{TateCurves}, for each $q \in \Delta_\iota^*$, there is a surjective homomorphism $\phi_{\iota,q} \colon \GG_m(\CC_\iota) \to \aE_{\iota,q}$ with kernel~$q^\ZZ$.
The morphism $\phi_{\iota,q}$ is defined by explicit series and we can verify that, if we substitute $q=0$ into these series, we obtain an isomorphism $\phi_{\iota,0} \colon \GG_{m,\CC_\iota} \to \aE_{\iota,0}$.
Furthermore, $\phi_{\iota,q}$ is (the map on $\CC_\iota$-points induced by) a morphism of complex or rigid analytic spaces for each~$q$ \cite[sec.~5.1]{FvdP04}.

Define an invariant differential form on $\aE_\iota/\Delta_\iota$ by
\[ \tilde\omega_\iota = \frac{dx}{2y+x} \in \Omega^1(\aE_\iota/\Delta_\iota). \]
Calculation, or \cite[p.~15]{TateCurves}, shows that, if $u$ denotes the standard coordinate on $\GG_m$, then
\begin{equation} \label{eqn:phi-omega}
\phi_{\iota,q}^*(\tilde\omega_{\iota,q}) = \frac{du}{u}.
\end{equation}
This should be interpreted as an equation of differential forms on complex or rigid analytic spaces over~$\CC_\iota$ (depending on whether $\CC_\iota$ is archimedean or non-archimedean).

\subsection{Map from the Tate family to the ``\texorpdfstring{$1/j$}{1/j}'' family}

Let $\Delta'_\iota = \Delta_\iota$ if $\iota$ is non-archimedean and $D(0, e^{-2\pi}, \CC_\iota)$ if $\CC_\iota$ is archimedean.
Note that $j^\iota$ does not take values $0$ or $1728$ on $\Delta_\iota'$.  Hence $1/j^\iota$ maps $\Delta_\iota'$ into $S^\iota(\CC_\iota)$ (this is the reason for choosing $\Delta_\iota'$ smaller than $\Delta_\iota$ for archimedean emebddings~$\iota$).

For each $q \in \Delta'_\iota \setminus \{0\}$, the elliptic curves $\aE_{\iota,q}$ and $\cE^\iota_{1/j^\iota(q)}$ have the same $j$-invariant, so they are isomorphic over~$\CC_\iota$.
We now show that we can choose these isomorphisms in such a way that the associated scaling factor on invariant differentials is given by evaluating a power series $\alpha \in \powerseries{\ZZ}{X}$, independent of $\iota$ or~$q$.

\begin{lemma} \label{map-to-j}
There exists
\begin{enumerate}[(a)]
\item a power series $\alpha \in \powerseries{\ZZ}{X}$, and
\item for each embedding $\iota$ of~$\QQ$ and each $q \in \Delta_\iota'$, a morphism of $\CC_\iota$-algebraic groups $\psi_{\iota,q} \colon \aE_{\iota,q} \to \cE^\iota_{1/j^\iota(q)}$,
\end{enumerate}
such that, for every embedding $\iota$ of $\QQ$,
\begin{enumerate}[(i)]
\item $\alpha^\iota$ converges on $\Delta_\iota'$, and
\item for each $q \in \Delta_\iota'$, we have $\psi_{\iota,q}^*(\omega_{1/j^\iota(q)}) = \alpha^\iota(q) \tilde\omega_{\iota,q}$.
\end{enumerate} 
\end{lemma}

\begin{remark}
The geometry which lies behind \cref{map-to-j} is that, for each~$\iota$, there is a morphism $\psi_\iota \colon \aE_\iota|_{\Delta_\iota'} \to \cE^{\iotaan}$ of complex or rigid analytic spaces, such that the following diagram commutes:
\[ \xymatrix{
    \aE_\iota|_{\Delta_\iota'}   \ar[r]^{\psi_\iota}  \ar[d]
  & \cE^\iotaan                  \ar[d]^{\pi^\iotaan}
\\  \Delta_\iota'                \ar[r]^{1/j^\iota}
  & S^\iotaan,
} \]
and such that, for every $\CC_\iota$-point $q \in \Delta_\iota'$, the restriction of~$\psi_\iota$ to~$\aE_{\iota,q}$ is equal to the morphism of complex or rigid analytic spaces underlying $\psi_{\iota,q}$.
This follows from equation~\eqref{eqn:isom-eq} in the proof below, and the fact that $u(q), r(q), s(q), t(q)$ in that equation are polynomials in $\alpha^\iota(q)^{-1}$, as can be seen using the first three lines of \cite[sec~III.1, Table~3.1]{SilAEC2ed}.
\end{remark}

\begin{proof}
The $q$-expansions of normalized Eisenstein series
$E_4(q) = 1 + 240s_3(q)$ and $E_6(q) = 1-504s_5(q)$ each have constant coefficient~$1$ and all their other coefficients are in $4\ZZ$.
Hence, the same is true of $E_6/E_4 \in \powerseries{\QQ}{q}$ and we can write
\[ E_6/E_4 = 1 + 4h, \]
where $h(q) \in \powerseries{\ZZ}{q}$ has no constant term.

The power series
\[ (1+4X)^{1/2} = 1 + \sum_{n=1}^\infty (-1)^{n+1} \frac{2}{n} \binom{2n-2}{n-1} X^n \]
has coefficients in $\ZZ$ because $\frac{1}{n} \binom{2n-2}{n-1}$ is a Catalan number.
Hence,
\[ \alpha(q) := (1+4X)^{1/2} \circ h = 1 + 372q + 127692q^2 + \dotsb \]
has coefficients in $\ZZ$.
Furthermore, $\alpha^2 = E_6/E_4$ in the ring $\powerseries{\QQ}{q}$.

Since $\alpha \in \ZZ[X]$, for every non-archimedean embedding $\iota$ of~$\QQ$, $\alpha^\iota$ converges on~$\Delta_\iota$.
For the archimedean embedding $\iota$ of~$\QQ$, recall that the modular forms $E_4^\iota(e^{2\pi i\tau})$, $E_6^\iota(e^{2\pi i\tau})$ are holomorphic on $\cH$ and that they have no zeros on $\cH$ with imaginary part greater than~$1$.
Consequently, $E_6^\iota(e^{2\pi i\tau})/E_4^\iota(e^{2\pi i\tau})$ is holomorphic and non-zero on $\cS = \{ \tau \in \cH : \Im(\tau) > 1 \}$, so $E_6^\iota/E_4^\iota$ is holomorphic and non-zero on the image of $\cS$ under $e^{2\pi i-}$, namely $\Delta'_\iota$.
It follows that $E_6^\iota/E_4^\iota$ has a holomorphic square root on this disc.
If we choose the square root of $E_6^\iota/E_4^\iota$ with the correct sign, its Taylor series will be $\alpha^\iota$, so $\alpha^\iota$ converges on $\Delta_\iota'$.

\medskip

For each embedding $\iota$ of~$\QQ$ and each $q \in \Delta_\iota'^*$, since $\aE_{\iota,q}$ and $\cE^\iota_{1/j^\iota(q)}$ have the same $j$-invariant, there is an isomorphism of elliptic curves $\psi_{\iota,q} \colon \aE_{\iota,q} \to \cE^\iota_{1/j^\iota(q)}$.
By \cite[Prop.~III.3.1(b)]{SilAEC2ed}, this isomorphism must have the form
\begin{equation} \label{eqn:isom-eq}
\psi_{\iota,q}(\tilde{x}, \tilde{y}) = (u(q)^2\tilde{x} + r(q), \, u(q)^3\tilde{y} + u(q)^2s(q)\tilde{x} + t(q))
\end{equation}
for some $u(q) \in \CC_\iota^\times$ and $r(q), s(q), t(q) \in \CC_\iota$.

The short Weierstrass coefficients for $\cE^\iota_{1/j^\iota(q)}$ and $\aE_{\iota,q}$ are
\begin{align*}
    c_4(1/j^\iota(q))  & = \frac{j^\iota(q)}{j^\iota(q)-1728},
  & \tilde{c}_4(q)     & = E_4^\iota(q),
\\  c_6(1/j^\iota(q))  & = -\frac{j^\iota(q)}{j^\iota(q)-1728},
  & \tilde{c}_6(q)     & = -E_6^\iota(q).
\end{align*}
From \cite[sec.~III.1, Table~3.1]{SilAEC2ed}, these short Weierstrass coefficients are related by
\[ u(q)^4\tilde{c}_4(q) = c_4(1/j^\iota(q)), \quad u(q)^6\tilde{c}_6(q) = c_6(1/j^\iota(q)). \]
Therefore,
\[ u(q)^2 = \frac{\tilde{c}_4(q) c_6(1/j^\iota(q))} {\tilde{c}_6(q) c_4(1/j^\iota(q))} = \frac{E_4^\iota(q)}{E_6^\iota(q)} = \alpha^\iota(q)^{-2}. \]
(Note that, as discussed above, $j^\iota$ does not take the values $0$ or $1728$ on $\Delta_\iota'$, so $c_4(1/j^\iota(q))$ and $c_6(1/j^\iota(q))$ are finite and non-zero on $\Delta_\iota'$ for any embedding $\iota$.
Since $u$ is non-zero on $\Delta_\iota'$, it follows that $\tilde{c}_4(q)$ and $\tilde{c}_6(q)$ are non-zero on $\Delta_\iota'$.)

Thus $u(q) = \pm\alpha^\iota(q)^{-1}$.
If $u(q) = -\alpha^\iota(q)^{-1}$, then we can replace the isomorphism $\psi_q$ by $[-1] \circ \psi_q$.
This replacement multiplies $u$ by $-1$.
Thus, we obtain $u(q) = \alpha^\iota(q)^{-1}$ for all $q \in \Delta'^*_\iota$.
Again from \cite[sec.~III.1, Table~3.1]{SilAEC2ed}, we read off that
\[ u(q)^{-1} \tilde\omega_{\iota,q} = \psi_{\iota,q}^*(\omega^\iota_{1/j^\iota(q)}), \]
proving (ii) for all $q \neq 0$.

For $q=0$, define $\psi_{\iota,0}$ to be the identity morphism $\GG_{m,\CC_\iota} \to \GG_{m,\CC_\iota}$.  Note that this is the same as the morphism that would be given by the formula \eqref{eqn:isom-eq} at $q=0$, after calculating $r(q)$, $s(q)$ and $t(q)$ as polynomials in $u(q)=\alpha^\iota(q)^{-1}$ using \cite[sec.~III.1, Table~3.1]{SilAEC2ed}.
Since $\alpha^\iota(0)=1$, (ii) is satisfied for $q=0$.
\end{proof}

\subsection{Periods} \label{subsec:periods-defs}

In order to define a power series~$F$ which will be play a central role throughout the paper, we modify $\alpha$ from \cref{map-to-j} so that its $\iota$-adic evaluations become functions of $s=1/j^\iota(q) \in S^\iota(\CC_\iota)$ instead of functions of $q \in \Delta'_\iota$. To accomplish this, note first that $1/j$ has no constant term and its coefficient of degree $1$ is a unit in $\ZZ$.
Hence $1/j$ has a compositional inverse in $\powerseries{\ZZ}{X}$, which we denote by
\[ \theta(X) = X + 744X^2 + 750420X^3 + \dotsb \in \powerseries{\ZZ}{X}. \]

Define
\[ F = \alpha \circ \theta \in \powerseries{\ZZ}{X}. \]

Since $\theta, \alpha \in \powerseries{\ZZ}{X}$, by \cref{padic-composition}, for all non-archimedean embeddings $\iota$ of $\QQ$, we have $R(F^\iota) \geq 1$ and for all $s \in \Delta_\iota$,
\begin{equation} \label{eqn:F-composition}
F^\iota(s) = \alpha^\iota(\theta^\iota(s)).
\end{equation}
Note that $\Delta_\iota \subset S^\iota(\CC_\iota)$ since $\abs{1/1728}_\iota \geq 1$ for all non-archimedean embeddings~$\iota$.

For the archimedean embedding $\iota$ of~$\QQ$, let $\Delta_S$ denote an open disc in $\CC$ centred at~$0$ with the following properties:
\begin{enumerate}
\item the radius of $\Delta_S$ is at most $1/1728$ (so $\Delta_S \subset S^\iota(\CC)$);
\item $\theta^\iota$ converges on $\Delta_S$;
\item $\theta^\iota(\Delta_S) \subset D(0, e^{-2\pi }, \CC)$.
\end{enumerate}
Then the disc of convergence of $F^\iota$ contains $\Delta_S$ and \eqref{eqn:F-composition} holds when $\iota$ is the archimedean embedding for all $s \in \Delta_S$.

\medskip

We can interpret the evaluation of $F$ at the archimedean embedding~$\iota \colon \QQ \to \CC$ as giving periods of elliptic curves over~$\CC$.
Indeed, let $\tilde\gamma$ be the generator of $H_1(\CC^\times, \bZ) \cong \bZ$ such that $\int_{\tilde\gamma} du/u = 2\pi i$ (this fixes an orientation of $\bC$).
Then, for $s \in \Delta_S$ and $q = \theta^\iota(s) \in \Delta'_\iota$,
\[ \gamma_s := \psi_{\iota,q*} \phi_{\iota,q*}(\tilde\gamma) \in H_1(\cE^\iota_s(\CC), \ZZ) \]
are the values of a section $\gamma \in (R_1 \pi^\iota_* \QQ)(\Delta_S)$.
In particular, for $s \in \Delta_S^*$, $\gamma_s$ is locally invariant (that is, invariant under the monodromy action of $\pi_1(\Delta_S^*, s)$).
Using \cref{map-to-j}(ii), \eqref{eqn:phi-omega} and \eqref{eqn:F-composition}, we have
\begin{equation} \label{eqn:arch-F-integral}
     \frac{1}{2\pi i} \int_{\gamma_s} \omega^\iota_s
   = \frac{1}{2\pi i} \int_{\tilde\gamma} \phi_{\iota,q}^*\psi_{\iota,q}^*(\omega^\iota_s)
   = \frac{1}{2\pi i} \int_{\tilde\gamma} \alpha^\iota(q) \frac{du}{u} = \alpha^\iota(q) = F^\iota(s).
\end{equation}

\medskip

Now we define three additional functions which, together with $F$, make up the full period matrix of the family $\cE$ under the archimedean embedding~$\iota \colon \QQ \to \CC$.
The first of these, $G$, is the archimedean evaluation of a power series with coefficients in $\QQ$; this matters because we need $G$ to be a G-function, but we will not make use of non-archimedean evaluations of~$G$.
The other two functions, $F^*$ and $G^*$, are defined only as holomorphic functions on a subset of $\Delta_S^*$ and are not evaluations of power series or even Laurent series at all (they are not meromorphic at~$0$).

Let $\eta \in H^1_{DR}(\cE|_{S^*}/S^*)$ denote the de Rham cohomology class which, on each fibre $\cE_s$, is represented by the differential form of the second kind
\[ (x+\tfrac{1}{12}) \, \frac{dx}{2y+x}. \]
For each $s \in S^*(K)$ (where $K$ is any field of characteristic~$0$), the de Rham classes $[\omega_s]$ and $\eta_s$ form a symplectic basis for $H^1_{DR}(\cE_s/K)$ with respect to the residue pairing.
Define a holomorphic function $G^\iota \colon \Delta_S \to \bC$ by
\begin{equation} \label{eqn:arch-G-integral}
G^\iota(s) = \frac{1}{2\pi i} \int_{\gamma_s} \eta^\iota_s.
\end{equation}

\begin{lemma}
$G^\iota$ is the complex evaluation of a power series $G \in \powerseries{\QQ}{X}$.
\end{lemma}

\begin{proof}
Let $\nabla$ denote the Gauss--Manin connection on $H^1_{DR}(\cE|_{S^*}/S^*)$.
Then $[\omega]$ and $\nabla_{d/ds}([\omega])$ form a $\QQ(S)$-basis for $H^1_{DR}(\cE_{\QQ(S)}/\QQ(S))$.
So $\eta = a[\omega] + b\nabla_{d/ds}([\omega])$ for some $a,b \in \QQ(S)$.
It follows that $G^\iota$ is the $\iota$-adic evaluation of
\[ G = aF + b\frac{dF}{dX} \in \powerseries{\QQ}{X}.
\qedhere \]
\end{proof}

\medskip

Let $\Delta_S^\dag$ be a simply connected open subset of $\Delta_S \setminus \{0\}$ which contains $S^*(\ov\QQ) \cap \Delta_S$.  For example, we can choose $\Delta_S^\dag$ to be the complement in $\Delta_S \setminus \{0\}$ of a ray from $0$ which misses all $\ov\QQ$-points.

Since $\Delta_S^\dag$ is simply connected, the local system $R_1 \pi_*^\iota \bQ|_{\Delta_S^\dag}$ is trivial.
This local system has rank~$2$, so we can choose a section $\delta$ of $R_1 \pi_*^\iota \bQ|_{\Delta_S^\dag}$ such that $\gamma$ and $\delta$ form an oriented basis for $(R_1 \pi_*^\iota \bQ)(\Delta_S^\dag)$.  (``Oriented'' means that $\gamma \cdot \delta = +1$, where $\cdot$ denotes the intersection pairing on the fibres.)
Define
\[ F^*(s) = \frac{1}{2\pi i} \int_{\delta_s} \omega^\iota_s, \quad G^*(s) = \frac{1}{2\pi i} \int_{\delta_s} \eta^\iota_s. \]
These are holomorphic functions on $\Delta_S^\dag$, but they do not extend to holomorphic functions on $\Delta_S^*$ because $\delta$ has non-trivial monodromy around $0$.

By the Legendre period relation for elliptic curves, for every $s \in \Delta_S^\dag$,
\begin{equation} \label{eqn:legendre-relation}
\det \fullmatrix{F^\iota(s)}{F^*(s)}{G^\iota(s)}{G^*(s)}
= \frac{1}{(2\pi i)^2} \det \fullmatrix{\int_{\gamma_s} \omega^\iota_s}{\int_{\delta_s} \omega^\iota_s}{\int_{\gamma_s} \eta^\iota_s}{\int_{\delta_s} \eta^\iota_s}  = \frac{1}{2\pi i}.
\end{equation}

\begin{remark}
This will not be used in this paper, but we can also interpret the $p$-adic evaluation of $F$ in terms of the $p$-adic period pairing
\[ \langle -, - \rangle \colon T_p(G) \times H^1_{DR}(G/\CC_p) \to B_{DR}, \]
defined for $G$ a commutative algebraic group over $\CC_p$.
Indeed, let $\iota$ denote the embedding $\QQ \to \CC_p$.
Choose a generator $\tilde\gamma_\iota$ for $T_p(\bG_m) = \bZ_p(1)$ (a ``$p$-adic orientation'').
Then the ``$p$-adic $2\pi i$''
\[ t_p := \langle \tilde\gamma_\iota, du/u \rangle \]
is a uniformiser of $B_{DR}^+$, and $\psi_{\iota,q*} \phi_{\iota,q*}(\tilde\gamma_\iota)$ is a locally invariant $p$-adic cycle on $\cE^\iota(\CC_\iota)|_{\Delta_\iota}$.
The same calculation as~\eqref{eqn:arch-F-integral} shows that, for all $s \in \Delta_\iota$ with $q = \theta^\iota(s)$,
\[ \frac{1}{t_p} \bigl\langle \psi_{\iota,q*} \phi_{\iota,q*}(\tilde\gamma_\iota), \omega^\iota_s \bigr\rangle
= \frac{1}{t_p} \bigl\langle \tilde\gamma_\iota, \psi_{\iota,q}^* \phi_{\iota,q}^*(\omega^\iota_s) \bigr\rangle
= \frac{1}{t_p} \bigl\langle \tilde\gamma_\iota, F^\iota(s) du/u \bigr\rangle
= F^\iota(s). \]
\end{remark}

\section{Period relations for isogenous points} \label{sec:period-relations}

In this section, we show how isogenies between elliptic curves over a number field $\hat K$ lead to polynomial relations between corresponding evaluations of $F$ and~$G$.
We construct these relations separately for archimedean and non-archimedean embeddings of~$\hat K$.

\begin{definition}
  If $N$ is a non-zero integer, we write $d(N)$ for the number of positive divisors of~$N$.  
\end{definition}

\begin{theorem} \label{global-relation}
Let $\hat K$ be a number field.
Let $M$, $N$ be positive integers.
Let $s_1, s_2, s_3, s_4 \in S^*(\hat K)$ be points such that:
\begin{enumerate}[(i)]
\item there is an isogeny $\cE_{s_1} \to \cE_{s_2}$ of degree~$M$ defined over $\hat K$;
\item there is an isogeny $\cE_{s_3} \to \cE_{s_4}$ of degree~$N$ defined over $\hat K$.
\end{enumerate}
Then:
\begin{enumerate}
\item There exists a non-zero homogeneous polynomial $P_\infty \in \hat K[Y_1, Z_1, \dotsc, Y_4, Z_4]$ of degree at most $2[\hat K:\QQ]$ such that, for every archimedean embedding $\iota$ of $\hat K$ satisfying $s_i^\iota \in \Delta_S$ for all~$i=1,2,3,4$, we have
\[ P_\infty^\iota(F^\iota(s_1^\iota), G^\iota(s_1^\iota), \dotsc, F^\iota(s_4^\iota), G^\iota(s_4^\iota)) = 0. \]
Furthermore, $P_\infty$ is not in the ideal $\langle Y_1-Y_3, Z_1-Z_3 \rangle$.
\item There exists a non-zero homogeneous polynomial $P_{fin} \in \hat K[Y_1, Y_2]$ of degree at most $2d(M)$ such that, for every non-archimedean embedding $\iota$ of $\hat K$ satisfying $\abs{s_i^\iota} < 1$ for~$i=1,2$, we have
\[ P_{fin}^\iota(F^\iota(s_1^\iota),F^\iota(s_2^\iota)) = 0. \]
\end{enumerate}
\end{theorem}

\begin{proof}
(1) Let $P_\infty$ be the product of the polynomials $P_\iota$ given by \cref{arch-relation}, as $\iota$ runs through the archimedean embeddings $\iota$ of $\hat K$.

(2) Let $P_{fin}$ be the product of the polynomials $P_m$ from \cref{non-arch-relation} (applied to $s_1$ and $s_2$), as $m$ runs through the divisors (positive and negative) of~$M$.
\end{proof}

\begin{remark}
For each archimedean embedding $\iota$, $s^\iota \in \Delta_S$ implies that $\abs{s^\iota} < \min(R(F^\iota), R(G^\iota))$.  Hence the evaluations in \cref{global-relation}(1) make sense.

Note that the relations at non-archimedean embeddings involve only $F$, not $G$.
This enables us to construct the relation without knowing a geometric interpretation of the $p$-adic evaluations of~$G$.
It also saves us the effort of estimating $R(G^\iota)$ which would otherwise be needed to control the set on which \cref{global-relation}(2) makes sense.

The non-archimedean relations involve only $s_1, s_2$, not $s_3, s_4$.
This is not important, except that it avoids the need to assert that $P_{fin} \not\in \langle Y_1-Y_3, Z_1-Z_3 \rangle$.

We may need to apply \cref{global-relation} with $s_1=s_3$.  
The assertion that $P_\infty \not\in \langle Y_1-Y_3, Z_1-Z_3 \rangle$ in \cref{global-relation} is needed to ensure that $P_\infty$ does not collapse to~$0$ upon making the substitution $Y_1=Y_3$, $Z_1=Z_3$.
\end{remark}

\begin{remark}
The polynomial $P_\infty$ is obtained by finding a relation for each archimedean embedding of $\hat K$ (\cref{arch-relation}), then multiplying these together.  There are $[\hat K:\QQ]$ archimedean embeddings, leading to a bound for $\deg(P_\infty)$.

The polynomial $P_{fin}$ is similarly constructed as a product, where one of the factors gives a relation between evaluations of $F$ at each non-archimedean embedding of $\hat K$ (\cref{non-arch-relation}).
However, there are infinitely many non-archimedean embeddings, so it is important that the relations at non-archimedean embeddings are not independent of each other: instead \cref{non-arch-relation} gives a finite collection of polynomials (of controlled size) whose members give relations between evaluations of~$F$ at every non-archimedean embedding where $s_1^\iota$ and $s_2^\iota$ are small enough.
\end{remark}

\subsection{Complex period relations}\label{sec:complex-rel}

To obtain a period relation at an archimedean embedding of $\hat K$, we require two isogenies $\cE_{s_1} \to \cE_{s_2}$ and $\cE_{s_3} \to \cE_{s_4}$.
A single isogeny gives a relation whose coefficients involve $2\pi i$, not just algebraic numbers.  A second isogeny gives a second such relation, allowing us to eliminate~$2\pi i$.

\begin{proposition} \label{arch-relation}
Let $\hat K$ be a number field.
Let $M$, $N$ be positive integers.
Let $s_1, s_2, s_3, s_4 \in S^*(\hat K)$ with the following properties:
\begin{enumerate}[(i)]
\item there is an isogeny $\cE_{s_1} \to \cE_{s_2}$ of degree~$M$ defined over $\hat K$;
\item there is an isogeny $\cE_{s_3} \to \cE_{s_4}$ of degree~$N$ defined over $\hat K$.
\end{enumerate}
Then, for each archimedean embedding $\iota$ of $\hat K$ satisfying $s_i^\iota \in \Delta_S$ for all~$i=1,2,3,4$, there exists a non-zero polynomial $P_\iota \in \hat K[Y_1, Z_1, \dotsc, Y_4, Z_4]$, homogeneous of degree~$2$, such that
\[ P_\iota^\iota(F^\iota(s_1^\iota), G^\iota(s_1^\iota), \dotsc, F^\iota(s_4^\iota), G^\iota(s_4^\iota)) = 0. \]
Furthermore, $P_\iota$ is not in the ideal $\langle Y_1-Y_3, Z_1-Z_3 \rangle$.
\end{proposition}

\begin{proof}
Let $\iota$ be an embedding $\hat K \to \CC$ satisfying $s_i^\iota \in \Delta_S$ for all~$i=1,2,3,4$.
To reduce notation, we omit mention of~$\iota$ for the rest of this proof.

We consider first the pair $s_1$, $s_2$, obtaining an inhomogeneous relation of degree at most $2$ between the $\iota$-adic evaluations of $F(s_1), G(s_1), F(s_2), G(s_2)$ whose value is a rational multiple of $1/2\pi i$.

Let $f \colon \cE_{s_1} \to \cE_{s_2}$ be an isogeny of degree~$M$.
Note that $f^*(\omega_{s_2}) \in \Omega^1(\cE_{s_1}/\hat K)$, which is a 1-dimensional $\hat K$-vector space, so there is an element $a \in \hat K$ such that
\begin{equation} \label{eqn:a-omega}
f^*(\omega_{s_2}) = a\omega_{s_1}.
\end{equation}
There are also $b, d \in \hat K$ such that
\[ f^*(\eta_{s_2}) = b[\omega_{s_1}] + d\eta_{s_1} \]
where $[\cdot]$ denotes the class of a differential form in $H^1_{DR}(\cE_{s_1}/\hat K)$.
Since $[\omega_s]$, $\eta_s$ form a symplectic basis of $H^1_{DR}(\cE_s/\hat K)$ for each~$s$, we have
\[ ad = \det \fullmatrix{a}{b}{0}{d} = M. \]

Looking at homology, there are $p, q, r, s \in \bZ$ such that
\begin{align*}
   f_*(\gamma_{s_1}) & = p\gamma_{s_2} + r\delta_{s_2},
\\ f_*(\delta_{s_1}) & = q\gamma_{s_2} + s\delta_{s_2}.
\end{align*}
Here $\det \fullmatrix{p}{q}{r}{s} = M$.

Since $s_1,s_2,s_3,s_4 \in \Delta_S$, we can apply \eqref{eqn:arch-F-integral} and \eqref{eqn:arch-G-integral}: the evaluations of power series $F(s_1)$, $G(s_1)$ are equal to periods of $\cE_{s_1}$, and so on for $s_2$, $s_3$, $s_4$.
Hence, using equations such as
\[ \int_{\gamma_{s_1}} f^*(\omega_{s_2}) = \int_{f_*(\gamma_{s_1})} \omega_{s_2}, \]
we obtain the period relations
\begin{equation} \label{eqn:c-period-matrix-relation}
\fullmatrix{a}{0}{b}{d} \fullmatrix{F(s_1)}{F^*(s_1)}{G(s_1)}{G^*(s_1)} = \fullmatrix{F(s_2)}{F^*(s_2)}{G(s_2)}{G^*(s_2)} \fullmatrix{p}{q}{r}{s}.
\end{equation}
Multiplying on the left by the row vector $(-G(s_2), F(s_2))$ and using \eqref{eqn:legendre-relation}, we obtain
\begin{align*}
  & \phantom{{}={}} \bigl( -aG(s_2) + bF(s_2), dF(s_2) \bigr) \fullmatrix{F(s_1)}{F^*(s_1)}{G(s_1)}{G^*(s_1)}
\\& = \bigl( -G(s_2)F(s_2) + F(s_2)G(s_2), -G(s_2)F^*(s_2) + F(s_2)G^*(s_2) \bigr) \fullmatrix{p}{q}{r}{s}.
\\& = \bigl( 0, 1/(2\pi i) \bigr) \fullmatrix{p}{q}{r}{s}.
\end{align*}
Multiplying on the right by the column vector $\binom{1}{0}$, we obtain
\begin{equation} \label{eqn:relation-pi}
-aG(s_2)F(s_1) + bF(s_2)F(s_1) + dF(s_2)G(s_1) = \frac{r}{2\pi i}.
\end{equation}

If $r=0$, then \eqref{eqn:relation-pi} is a relation of degree~$2$ with coefficients in $\hat K$ between the $\iota$-adic evaluations of power series $F(s_1)$, $F(s_2)$, $G(s_1)$, $G(s_2)$.
Note that this relation is non-zero because $ad=M \neq 0$.

If $r \neq 0$, then the $2\pi i$ on the right hand side means that \eqref{eqn:relation-pi} does not directly give us a relation between evaluations of $F$ and $G$.
Instead, we apply the same argument as above to $s_3$ and $s_4$, obtaining a relation
\begin{equation} \label{eqn:relation'-pi}
-a'G(s_4)F(s_3) + b'F(s_4)F(s_3) + d'F(s_4)G(s_3) = \frac{r'}{2\pi i}
\end{equation}
where $a', b', d' \in \hat K$ and $r' \in \bZ$.
Combining \eqref{eqn:relation-pi} and \eqref{eqn:relation'-pi}, we obtain
\begin{multline} \label{eqn:relation-quadratic}
  r' \bigl( -aG(s_2)F(s_1) + bF(s_2)F(s_1) + dF(s_2)G(s_1) \bigr)
\\ \mathbin{+} r \bigl( a'G(s_4)F(s_3) - b'F(s_4)F(s_3) - d'F(s_4)G(s_3) \bigr) = 0.
\end{multline}

In other words, the following polynomial in $\hat K[Y_1, Z_1, \dotsc, Y_4, Z_4]$ vanishes at $(F(s_1), G(s_1), \dotsc, F(s_4), G(s_4))$:
\[ P_\iota = r'(-aZ_2Y_1 + bY_2Y_1 + dY_2Z_1) + r(a'Z_4Y_3 - b'Y_4Y_3 - d'Y_4Z_3). \]
Each of the terms in the above expression involves a distinct monomial.
Since we are assuming that $r \neq 0$ and $a'd' = N \neq 0$, the term $ra'Z_4Y_3$ is non-zero.
Hence, this is a non-zero polynomial.

Furthermore, the monomials in the above expression remain distinct after the substitution $Y_1=Y_3$, $Z_1=Z_3$, so $P_\iota \not\in \langle Y_1-Y_3, Z_1-Z_3 \rangle$.
\end{proof}

\begin{remark}
If $r=0$, then we can obtain a relation of degree~$1$ directly from the top left corner of~\eqref{eqn:c-period-matrix-relation}, namely, $aF(S_1) = pF(s_2)$.
Thus $r=0$ corresponds to Case~2 in \cite[X, 2.4]{And89}, while $r \neq 0$ corresponds to Case~3, leading to a quadratic relation and requiring two inhomogeneous relations to eliminate a multiple of $1/2\pi i$.
\end{remark}

\subsection{Non-archimedean period relations}\label{sec:padic-rel}

In order to prove \cref{global-relation}(2), we shall use the following result about isogenies and the Tate uniformisation.

\begin{proposition} \label{isogeny-lifts-tate}
Let $p$ be a prime number and let $\iota$ be the embedding $\QQ \to \CC_p$.  Let $q_1, q_2 \in \Delta_\iota$ be such that there is an isogeny $\tilde{f} \colon \aE_{\iota,q_1} \to \aE_{\iota,q_2}$ of degree~$M$.
There is an integer $m$ dividing $M$ such that the following diagram (in the category of rigid spaces over~$\CC_p$) commutes:
\[ \xymatrix{
    \bG_m^\iotaan    \ar[d]^{\phi_{\iota,q_1}} \ar[r]^{[m]}
  & \bG_m^\iotaan    \ar[d]^{\phi_{\iota,q_2}}
\\  \aE_{\iota,q_1}                              \ar[r]^{\tilde{f}}
  & \aE_{\iota,q_2}
} \]
\end{proposition}

\begin{proof}
According to \cite[``Isogenies'', Theorem]{TateCurves}, 
every isogeny between $p$-adic Tate curves is of the form $\alpha_{m,n} \colon \aE_{\iota,q_1} \to \aE_{\iota,q_2}$ for some integers $m$ and $n$ such that $q_1^m = q_2^n$, where $\alpha_{n,m}$ fits into the following commutative diagram:
\[ \xymatrix@C+2em{
    0               \ar[r]
  & \bZ             \ar[r]^-{1 \mapsto q_1} \ar[d]^{[n]}
  & \GG_m^\iotaan   \ar[r]^-{\phi_{\iota,q_1}}  \ar[d]^{[m]}
  & \aE_{\iota,q_1}     \ar[r]                 \ar[d]^{\alpha_{n,m}}
  & 0
\\  0               \ar[r]
  & \bZ             \ar[r]^-{1 \mapsto q_2}
  & \GG_m^\iotaan   \ar[r]^-{\phi_{\iota,q_2}}
  & \aE_{\iota,q_2}     \ar[r]
  & 0
} \]
(note that our $m$ and $n$ are the opposite way round to those in \cite{TateCurves}).
Furthermore, the theorem in \cite{TateCurves} also tells us that $mn = M$ so $m$ divides $M$.
\end{proof}

\begin{proposition} \label{non-arch-relation}
Let $\hat K$ be a number field.
Let $s_1, s_2 \in S^*(\hat K)$ be such that there exists an isogeny $\cE_{s_1} \to \cE_{s_2}$ of degree~$M$ defined over $\hat K$.
Then there exists a set of non-zero homogeneous linear polynomials $P_m \in \hat K[Y_1,Y_2]$, indexed by the positive and negative integers $m$ dividing~$M$, such that, for every non-archimedean embedding $\iota$ of $\hat K$ satisfying $\abs{s_i^\iota} < 1$ for~$i=1,2$, there is some divisor $m$ of $M$ for which
\[ P_m^\iota(F^\iota(s_1^\iota), F^\iota(s_2^\iota)) = 0. \]
\end{proposition}

\begin{proof}
Let $f \colon \cE_{s_1} \to \cE_{s_2}$ be an isogeny of degree~$M$.
As in \eqref{eqn:a-omega}, let $a \in \hat K$ be the scalar such that $f^*(\omega_{s_2}) = a\omega_{s_1}$.  Note that $a$ is independent of the embedding~$\iota$.

For each divisor $m$ of $M$ in $\bZ$, let
\begin{equation} \label{eqn:pm}
P_m(Y_1,Y_2) = aY_1 - mY_2 \in \hat K[Y_1, Y_2].
\end{equation}
We will show that these are the polynomials we need for the conclusion of \cref{non-arch-relation}.

Let $\iota \colon \hat K \to \CC_\iota$ be an embedding such that $s_1^\iota, s_2^\iota \in \Delta_\iota$.
Let $q_1 = \theta^\iota(s_1), q_2 = \theta^\iota(s_2) \in \Delta_\iota$.

Since $\psi_{\iota,q}$ is an isomorphism of elliptic curves, there is a unique isogeny $\tilde{f}_\iota \colon \aE_{\iota,q_1} \to \aE_{\iota,q_2}$ of elliptic curves over $\CC_\iota$ such that the lower square of the following diagram commutes.
Applying \cref{isogeny-lifts-tate} to $\tilde{f}_\iota$, we obtain an integer $m$ dividing $M$ such that the upper square of the diagram commutes.
\begin{equation} \label{eqn:diagram-tilde-f}
\begin{tikzcd}
    \bG_m^\iotaan           \ar[d, "\phi_{\iota,q_1}"] \ar[r, dotted, "{[m]}"]
  & \bG_m^\iotaan           \ar[d, "\phi_{\iota,q_2}"]
\\  \aE_{\iota,q_1}         \ar[d, "\psi_{\iota,q_1}"] \ar[r, dotted, "\tilde{f}_\iota"]
  & \aE_{\iota,q_2}         \ar[d, "\psi_{\iota,q_2}"]
\\  \cE_{s_1}^\iotaan       \ar[r, "f^\iotaan"]
  & \cE_{s_2}^\iotaan
\end{tikzcd}
\end{equation}

We consider~\eqref{eqn:diagram-tilde-f} as a commutative diagram in the category of rigid analytic spaces over~$\CC_\iota$.

Pulling back $\tilde\omega_{\iota,q_2}$ around the upper square of diagram~\eqref{eqn:diagram-tilde-f}, we obtain
\[ \phi_{\iota,q_1}^* \tilde{f}_\iota^{*}(\tilde\omega_{\iota,q_2})
   = [m]^* \phi_{\iota,q_2}^*(\tilde\omega_{\iota,q_2})
   = m\frac{du}{u} = m \, \phi_{\iota,q_1}^*(\tilde\omega_{\iota,q_1}). \]
Since $\phi_{\iota,q_1}^* \colon \Omega^1(\aE_{\iota,q_1}/\CC_\iota) \to \Omega^1(\bG_m/\CC_\iota)$ is injective, it follows that
\[ \tilde{f}_\iota^{*}(\tilde\omega_{\iota,q_2}) = m \, \tilde\omega_{\iota,q_1}. \]
Using this together with \eqref{eqn:F-composition}, \eqref{eqn:a-omega} and the lower square of~\eqref{eqn:diagram-tilde-f}, we can calculate
\begin{align*}
      \tilde{f}_\iota^{*} \psi_{\iota,q_2}^*(\omega^\iota_{s_2})
  & = \tilde{f}_\iota^{*} \bigl( F^\iota(s_2^\iota) \tilde\omega_{\iota,q_2} \bigr)
    = F^\iota(s_2^\iota) m \, \tilde\omega_{\iota,q_1}
\\& = \psi_{\iota,q_1}^* f_\iota^*(\omega^\iota_{s_2})
    = \psi_{\iota,q_1}^*(a \omega^\iota_{s_1})
    = a F^\iota(s_1^\iota) \tilde\omega_{\iota,q_1}.
\end{align*}

Since $\tilde\omega_{\iota,q_1} \neq 0$, we conclude that
\begin{equation} \label{eqn:padic-period-relation}
aF^\iota(s_1^\iota) = mF^\iota(s_2^\iota).
\end{equation}
In other words, \eqref{eqn:pm} gives a relation between the $\iota$-adic evaluations of $F(s_1)$ and $F(s_2)$, as required.
\end{proof}

\section{Proof of lower bound for Galois orbits}\label{sec:proofs}

In this section we prove the height bound \cref{andre-bound} and deduce the Galois orbits bound \cref{galois-bound}.
The proof follows the strategy of \cite[Ch.~X]{And89}, using the polynomial relations between periods of fibres with exceptional isogenies obtained in \cref{global-relation}.
In particular, the work of finding relations between non-archimedean evaluations of $F$ was already carried out in \cref{global-relation}(2).
Thus the only new work in this section needed to make use of non-archimedean evaluations of G-functions is to pay attention to non-archimedean convergence of power series.

In order to be confident that the argument is valid with non-archimedean evaluations, we have written the full argument in detail. This includes expanding the remark on \cite[p.~202]{And89} about finite ramification at $s_0$ when choosing the local parameter, via \cref{curve-ramification}.
To help the reader follow the essential structure of the argument, we have included in Section~\ref{subsec:introductory-case} a summary for a special case in which many of the technical difficulties do not arise.

We found Papas's paper \cite{Papas} very helpful when writing this section, especially for describing the functional relations using André's Normal Monodromy Theorem \cite{And92}, which is simpler than the approach in \cite[X, Lemma~3.3]{And89}.

Please note that the notation $s_1, \dotsc, s_k$ in this section will refer to a collection of base points in the curve~$C_4$ (as defined two paragraphs above \cref{sigma-ell}).
This differs from section~\ref{sec:period-relations}, where $s_1,s_2,s_3,s_4$ referred to reciprocals of $j$-invariants of elliptic curves (in other words, reciprocals of the coordinates of a point in~$Y(1)^4$).

\subsection{A special case} \label{subsec:introductory-case}

In order to illustrate our strategy for the proof of \cref{andre-bound}, we consider a special case in which $\ov{C}$ (the Zariski closure of $C$ in $(\PP^1)^n$) is isomorphic to $\PP^1$.
This demonstrates the central idea of how we exploit the global relations given by \cref{global-relation}, while avoiding technical complications which appear in the general case.
We shall not give full details of all steps in this special case, because full details are given in the general case.

Let $s_0 = (\infty, \dotsc, \infty) \in \ov{C}$.
Choose an isomorphism $x \colon \ov{C} \to \PP^1$ such that $x(s_0) = 0$.
For each $i=1, \dotsc, n$, let $\rho_i \colon \ov{C} \to \PP^1$ denote the reciprocal of the projection onto the $i$-th coordinate (note that $\rho_i(s_0)=0$ for all $i$). Let $\xi_i$ denote the Taylor series of $\rho_i$ in terms of the local parameter $x$ around $s_0$.
Define power series $F_i, G_i \in \powerseries{K}{X}$ by
\[ F_i = F \circ \xi_i, \quad G_i = G \circ \xi_i. \]
For every embedding $\iota$ of $K$, on a suitable neighbourhood of $s_0^\iota$ in $\ov{C}^\iota(\CC_\iota)$, we have
\[ F_i^\iota(x^\iota(s)) = F^\iota(\rho_i^\iota(s)), \]
and similarly for~$G$ (diagram~\eqref{eqn:F-xi-diagram}).

Restricting our attention to $C_0 = \bigcap_i \rho_i^{-1}(S) \subset C$ (throwing away only finitely many points), we may consider the semiabelian schemes $\rho_i^* \cE$ over~$C_0$.
We can check that, for each archimedean embedding $\iota$ of $K$ and on a suitable disc around $s_0^\iota$ in $\ov{C}^\iota(\CC)$, the $\iota$-adic evaluations of $F_i$ and $G_i$ give locally invariant periods of $(\rho_i^* \cE)^\iota$ (equation~\eqref{eqn:arch-Fli-integral}).
It follows that $F_i$ and $G_i$ are G-functions (\cref{G-functions}).

Furthermore, since $C$ is Hodge generic, the semiabelian schemes $\rho_i^* \cE$ are not geometrically generically isogenous.  This implies that there are no non-zero functional relations between the G-functions $\{ F_i, G_i : 1 \leq i \leq n \}$ (\cref{prop:trivial}).

Let $s \in C(\ov\QQ)$ satisfy the hypothesis of \cref{andre-bound}, and let $\hat K = K(s)$.
Since $\Phi_M(s_{i_1}, s_{i_2}) = 0$, there is an isogeny $\cE_{\rho_{i_1}(s)} \to \cE_{\rho_{i_2}(s)}$ of degree~$M$.
Similarly, there is an isogeny $\cE_{\rho_{i_3}(s)} \to \cE_{\rho_{i_4}(s)}$ of degree~$N$.

Applying \cref{global-relation}, we obtain homogeneous polynomials $P_\infty,P_{fin}$ such that, for all embeddings $\hat\iota$ of $\hat K$ in which $s^{\hat\iota}$ is close to $s_0^{\hat\iota}$, we have
(if $\hat\iota$ is archimedean)
\begin{align*}
    0
  & = P_\infty^\iota(F^\iota(\rho_{i_1}(s)^{\hat\iota}), G^\iota(\rho_{i_1}(s)^{\hat\iota}), \dotsc, F^\iota(\rho_{i_4}(s)^{\hat\iota}), G^\iota(\rho_{i_4}(s)^{\hat\iota}))
\\& =P_\infty^\iota(F_{i_1}^\iota(x(s)^{\hat\iota}), G_{i_1}^\iota(x(s)^{\hat\iota}), \dotsc, F_{i_4}^\iota(x(s)^{\hat\iota}), G_{i_4}^\iota(x(s)^{\hat\iota}))
\end{align*}
and (if $\hat\iota$ is non-archimedean)
\begin{align*}
    0
  & = P_{fin}^\iota(F^\iota(\rho_{i_1}(s)^{\hat\iota}), F^\iota(\rho_{i_2}(s)^{\hat\iota}))
    = P_{fin}^\iota(F_{i_1}^\iota(x(s)^{\hat\iota}), F_{i_2}^\iota(x(s)^{\hat\iota})).
\end{align*}
Define a polynomial $Q \in \hat K[Y_1, Z_1, \dotsc, Y_n, Z_n]$ as follows:
\[ Q(\underline{Y}, \underline{Z}) = P_\infty(Y_{i_1}, Z_{i_1}, \dotsc, Y_{i_4}, Z_{i_4}) \cdot P_{fin}(Y_{i_1}, Y_{i_2}). \]
Then
\[ Q^\iota(F_i^\iota(x(s)^{\hat\iota}), G_i^\iota(x(s)^{\hat\iota}) : 1 \leq i \leq n) = 0 \]
for every embedding $\hat\iota$ at which $x^\iota(s^{\hat\iota})$ is small (since the only zero of $x$ in $\ov{C}$ is at $s_0$, this is equivalent to $s^{\hat\iota}$ being close to $s_0^{\hat\iota}$).

By paying careful attention to the radii within which the above manipulations of power series are valid, we can conclude that $Q$ is a global relation between the evaluations at $x(s)$ of the G-functions $F_1, G_1, \dotsc, F_n, G_n$, along with a rational function which we add to the set of G-functions in order to control the radii of convergence in the definition of global relation (\cref{global-relation-applied}).

By consideration of the possible coincidences between the indices $i_1, i_2, i_3, i_4$, and using the facts that $P_\infty \not\in \langle Y_1-Y_3, Z_1-Z_3 \rangle$ and $P_{fin} \neq 0$, we can ensure that $Q \neq 0$, and thus $Q$ is not a trivial relation.

From the degree bounds in \cref{global-relation}, we have
\[ \deg(Q) \leq 2[\hat K:\QQ] + 2d(M) \ll \max\{ [K(s):K], M^\epsilon \}, \]
using the fact that, for every $\epsilon > 0$,
\begin{equation} \label{eqn:dN-bound}
d(N) \ll_\epsilon N^\epsilon
\end{equation}
(see \cite[Theorem~315]{HardyWright}).
We can thus complete the proof of this special case of \cref{andre-bound} by applying \cref{hasse-principle-E}.

\subsection{Outline of technical issues in the general case} \label{subsec:technical-issues}

The proof of a special case of \cref{andre-bound} in Section~\ref{subsec:introductory-case} made use of two properties of the morphism $x \colon \ov{C} \to \PP^1$: the only zero of $x$ was at $s_0=(\infty, \dotsc, \infty)$ and $x$ was a local parameter at $s_0$.
A morphism with these properties must have degree~$1$, so can only exist when $\ov{C}$ is a rational curve.

In order to prove \cref{andre-bound} when $\ov{C}$ is not rational, we will replace $\ov{C}$ by another curve $C_4$, equipped with a surjective morphism $C_4 \to \ov{C}$ and with a finite set of points $s_1, \dotsc, s_\ell \in C_4$ mapping to $s_0 \in \ov{C}$.  We will choose $C_4$ so that there is a morphism $x \colon C_4 \to \PP^1$ which vanishes only at $s_1, \dotsc, s_\ell$ and which is a local parameter at each of $s_1, \dotsc, s_\ell$.
This comes as close as we can to the properties of $x$ used in Section~\ref{subsec:introductory-case}, at the cost of having finitely many degeneration points $s_1, \dotsc, s_\ell$ instead of the single point $s_0$.

If $s \in C_4(\hat K)$ and $x(s)^{\hat\iota}$ is small for some embedding $\hat\iota$ of $\hat K$, then $s^{\hat\iota}$ is close to one of the $s_k^{\hat\iota}$.  However, even for a fixed $\hat K$-point~$s$, $s^\iota$ might be close to $s_k^{\hat\iota}$ for different~$k$ depending on the embedding~$\hat\iota$.  We therefore have to look for relations between locally invariant periods of the semiabelian schemes $(\rho_i^*\cE)^{\hat\iota}$ near each of $s_1^{\hat\iota}, \dotsc, s_\ell^{\hat\iota}$ (where $\rho_i$ denotes the reciprocal of a coordinate on $\ov{C}$, composed with the map $C_4 \to \ov{C}$).

Working with periods around different degeneration points is awkward, especially when it comes to understanding functional relations between the period functions.  We therefore impose an extra condition on $C_4$: for every $k=1,\dotsc,\ell$, $C_4$ has an automorphism such that $\sigma_k(s_1) = s_k$.
We can then work with the semiabelian scheme $\sigma_k^*\rho_i^*\cE$ in a neighbourhood of $s_1$ instead of $\rho_i^*\cE$ in a neighbourhood of $s_k$.

There may exist isogenies between the semiabelian schemes $\sigma_k^*\rho_i^*\cE$ (over a geometric generic point of the base), leading to functional relations between the associated period functions.
We therefore choose one semiabelian scheme from each isogeny class among the $\sigma_k^*\rho_i^*\cE$.
These are the semiabelian schemes which we will call $\cA_\lambda$.
Their locally invariant periods around $s_1$ give us a collection of G-functions with no functional relations.

After the technical setup summarised above, we will finally be able to complete the proof by finding global relations between evaluations of these G-functions at exceptional points along the lines explained in Section~\ref{subsec:introductory-case}.

\subsection{Curves and local parameters}

As explained in Section~\ref{subsec:technical-issues}, we need to replace $\ov{C}$ by a curve $C_4$ with a suitable rational function $x$ on $C_4$ which we will use as a local parameter at each of finitely many degeneration points.
The following lemma constructs $C_4$ and $x$ with the required properties.

\begin{lemma} \label{curve-ramification}
Let $C$ be an irreducible projective algebraic curve over a field $K$ of characteristic zero.
Let $s_0$ be a closed point of~$C$.
Then, after replacing $K$ by a finite extension, there exists an irreducible smooth projective algebraic curve $C_4$ over $K$, a non-constant morphism $\nu \colon C_4 \to C$ and a non-constant rational function $x \in K(C_4)$ such that
\begin{enumerate}[(i)]
\item every point $s \in C_4(\ov K)$ where $x(s)=0$ is a simple zero of $x$;
\item every point $s \in C_4(\ov K)$ where $x(s)=0$ satisfies $\nu(s)=s_0$;
\item $x \colon C_4 \to \PP^1$ is Galois, that is, $\Aut_x(C_4) := \{ \sigma \in \Aut(C_4) : x \circ \sigma = x \}$ acts transitively on each $\ov K$-fibre of $x$.
\end{enumerate}
\end{lemma}

\begin{proof}
Let $\nu_1 \colon C_1 \to C$ denote the normalisation of $C$.
Let $t$ be a point of $C_1$ such that $\nu_1(t)=s_0$.
(There may be more than one such point; choose one of them.)

By the Riemann--Roch theorem, there exists a rational function on~$C_1$ which has a pole at $t$ and nowhere else.  Taking the reciprocal of this function, we obtain a rational function~$x_1 \in K(C_1)$ with a zero at~$t$ and nowhere else.
Let $d\in\bN$ denote the order of vanishing of $x_1$ at $t$.

Let $C_2$ denote an irreducible component of the fibre product of $x_1 \colon C_1 \to \PP^1$ with the $d$-th power map $[d] \colon \PP^1 \to \PP^1$.
Let $\nu_2 \colon C_3 \to C_2$ denote the normalisation of $C_2$ and let $x_3$ and $\kappa$ be the morphisms indicated in the following commutative diagram:
\[ \begin{tikzcd}
    C_3 \arrow[r, swap, "\nu_2"] \arrow[rd, swap, "x_3"]
  & C_2 \arrow[r, swap, "\kappa"] \arrow[d]
  & C_1 \arrow[d, "x_1"] \arrow[r, swap, "\nu_1"]
  & C
\\& \bP^1 \arrow[r, "{[d]}"]
  & \bP^1
  &
\end{tikzcd} \]

After replacing $K$ by a finite extension, we may construct a smooth projective curve $C_4$ over~$K$ with a morphism $\nu_3 \colon C_4 \to C_3$ which is the Galois closure of $x_3 \colon C_3 \to \PP^1$.
Let $\nu$ and $x$ be the morphisms indicated in the following diagram:
\begin{equation} \label{eqn:curves-diagram}
\begin{tikzcd}
    C_4 \arrow[r, swap, "\nu_3"] \arrow[rrd, swap, "x"] \arrow[rrrr, bend left=20, "\nu"]
  & C_3 \arrow[r, swap, "\nu_2"] \arrow[rd]
  & C_2 \arrow[r, swap, "\kappa"] \arrow[d]
  & C_1 \arrow[d, "x_1"] \arrow[r, swap, "\nu_1"]
  & C
\\&& \bP^1 \arrow[r, "{[d]}"]
  & \bP^1
  &
\end{tikzcd}
\end{equation}

Consider a point $s \in C_4(\ov K)$ such that $x(s) = 0$.
Then $x_1\kappa\nu_2\nu_3(s) = [d]x(s) = 0$.
Hence by the choice of $x_1$, $\kappa\nu_2\nu_3(s)=t$ so $\nu(s)=\nu_1(t)=s_0$, establishing (ii).
Furthermore~(iii) holds because we chose $C_4$ to be the Galois closure of $x_3 \colon C_3 \to \PP^1$.

Since $t$ is the only zero of $x_1$ in $\PP^1$, and since the ramification degree of $x_1 \colon C_1 \to \PP^1$ at~$t$ is equal to the ramification degree of $[d] \colon \PP^1 \to \PP^1$ at~$0$, by Abhyankar's lemma \cite[\href{https://stacks.math.columbia.edu/tag/0BRM}{Tag 0BRM}]{stacks}, $x_3 \colon C_3 \to \PP^1$ is unramified at every zero of $x_3$ in $C_3$.
(Note that the function field $K(C_3)=K(C_2)$ is a direct factor of $K(C_1) \otimes_{K(\PP^1),[d]} K(\PP^1)$.)
Passing to the Galois closure, it follows that $x \colon C_4 \to \PP^1$  is unramified at every zero of $x$ in $C_4$, proving~(i).
\end{proof}

In the setting of \cref{andre-bound}, $C \subset \AAA^n$ is a geometrically irreducible Hodge generic curve defined over a number field~$K$ that intersects infinity.
Let $\overline{C}$ denote the Zariski closure of $C$ in $(\bP^1)^n$ and let $s_0 = (\infty, \dotsc, \infty) \in \overline{C}$.

Applying \cref{curve-ramification} to $\overline{C}$, after replacing $K$ by a finite extension, we obtain a smooth projective algebraic curve $C_4$ with a morphism $\nu \colon C_4 \to \overline{C}$ and a rational function $x \in K(C_4)$.
Replace the field $K$ by a further finite extension so that all zeros of $x$ in $C_4$ are defined over $K$.
Label the zeros of $x$ in $C_4(K)$ as $s_1, \dotsc, s_\ell$.
Since these zeros are simple, $x$ is a local parameter for $C_4$ around $s_k$ (which is to say, $x$ is a uniformiser in the local ring $\cO_{C_4,s_k}$) for all $k=1,\ldots,\ell$.

Let $\zeta_1, \dotsc, \zeta_n \colon \overline{C} \to \PP^1$ denote the reciprocals of the projections onto the $\PP^1$ factors.
For each $i=1,\dotsc,n$, let $\rho_i = \zeta_i \circ \nu \colon C_4 \to \PP^1$.

\begin{remark} \label{sigma-ell}
By \cref{curve-ramification}(iii), for each $k=1,\dotsc,\ell$, there is an automorphism $\sigma_k \in \Aut_x(C_4)$ such that $\sigma_k(s_1) = s_k$.

Since $x$ is unramified at $s_1, \dotsc, s_\ell$, we have $\deg(x)=\ell$.
Since $x$ is Galois, $\#\Aut_x(C_4) = \deg(x)$.
It follows that $\sigma_1, \dotsc, \sigma_\ell$ are the only elements of $\Aut_x(C_4)$.
\end{remark}

\subsection{Semiabelian schemes}\label{subsec:semiab}

We wish to pull back the $1/j$-family of elliptic curves $\cE$ to $C_4$ via $\rho_i$.
Note that $\cE$ is defined only over $S = \bA^1 \setminus \{1/1728\}$, not over $\PP^1$, so we must replace $C_4$ by a Zariski open subset in order to be able to do this.
Define
\[ C_0 = \{ s \in C_4 : \rho_i(\sigma_k(s)) \in S \text{ for all } k=1,\dotsc,\ell, \, i=1,\dotsc,n \}. \]
Similarly define
\[ C_0^* = \{ s \in C_4 : \rho_i(\sigma_k(s)) \in S^* \text{ for all } k=1,\dotsc,\ell, \, i=1,\dotsc,n \}. \]

Since the morphisms $\rho_i\circ\sigma_k$ are non-constant, each of $C_0$ and $C_0^*$ is the complement of finitely many points in $C_4$, so they are non-empty Zariski open subsets of $C_4$.
Note also that $\sigma_k$ restricts to an automorphism of $C_0$ for each $k=1, \dotsc, \ell$.

\medskip

By the definition of $C_0$, we have $\rho_i(C_0) \subset S$.
We can therefore define $\cA_i = (\rho_i|_{C_0})^*\cE$.
In other words, $\cA_i$ is a semiabelian scheme over $C_0$ which fits into the following commutative diagram:
\begin{equation} \label{eqn:Ai-diagram}
\begin{tikzcd}
    \cA_i   \arrow[d, "\pi_i"] \arrow[rr, "\theta_i"]
  &
  & \cE     \arrow[d, "\pi"]
\\  C_0     \arrow[rr, "\rho_i|_{C_0}"] \arrow[d, phantom, sloped, "\subset"]
  &
  & S       \arrow[d, phantom, sloped, "\subset"]
\\  C_4     \arrow[r, "\nu"]
  & \ov C   \arrow[r, "\zeta_i"]
  & \PP^1
\end{tikzcd}
\end{equation}
Since $\rho_i(C_0^*) \subset S^*$, the fibres of $\cA_i$ over $C_0^*$ are elliptic curves.
Meanwhile the fibres over $s_1, \dotsc, s_\ell \in C_0$ are isomorphic to $\bG_m$, because $\rho_i(s_k) = 0$ for all $i$ and $k$.

There may exist points $s \in C_0 \setminus (C_0^* \cup \{ s_1, \dotsc, s_\ell \})$.
For such points, the fibres $\cA_{i,s}$ may be elliptic curves for some~$i$ and isomorphic to~$\bG_m$ for other~$i$.  Such points will not appear further in our analysis.

\medskip

As explained in Section~\ref{subsec:technical-issues}, we will need to analyse periods of $\cA_i$ near the points $s_1, \dotsc, s_\ell$ (where ``near'' is interpreted with respect to an embedding of $K$).
It is easier to compare periods of a larger number of semiabelian schemes over a neighbourhood of a single point, so we will instead consider the semiabelian schemes $\sigma_k^* \cA_i$ ($1 \leq i \leq n$, $1 \leq k \leq \ell$), always near the point~$s_1$.
The periods of $\sigma_k^* \cA_i$ near $s_1$ are the same as the periods of $\cA_i$ near $s_k$.
In order to avoid functional relations between the periods of the schemes $\sigma_k^*\cA_i$, we pick out one of these schemes in each generic isogeny class.

More precisely, let $\bar\eta$ denote a geometric generic point of $C_0$.
Define an equivalence relation $\sim$ on $\{ 1, \dotsc, \ell \} \times \{ 1, \dotsc, n \}$ by
\[ (k,i) \sim (k',i') \text{ if there exists an isogeny } \sigma_k^*\cA_{i,\bar\eta} \to \sigma_{k'}^* \cA_{i',\bar\eta}. \]
Let $\Lambda$ denote a set of representatives of the equivalence classes for~$\sim$.

We define some convenient notation:
\begin{itemize}
\item Given $(k,i) \in \{ 1, \dotsc, \ell \} \times \{ 1, \dotsc, n \}$, let $[k,i]$ be the unique element of $\Lambda$ such that $[k,i] \sim (k,i)$.
\item If $\lambda = (k,i) \in \Lambda$, we shall write $\cA_\lambda = \sigma_k^* \cA_i$ and $\rho_\lambda = \rho_i \circ \sigma_k \colon C_0 \to \PP^1$.
\end{itemize}

Note that combining these two pieces of notation leads to $\cA_{[k,i]} = \sigma_{k'}^*\cA_{i'}$ where $[k,i] = (k',i')$.  Thus if $(k,i) \not\in \Lambda$, $\cA_{[k,i]}$ need not be equal to $\sigma_k^*\cA_i$.
However $\cA_{[k,i],\bar\eta}$ is always isogenous to $\sigma_k^*\cA_{i,\bar\eta}$.
Thanks to this isogeny, periods of $\cA_{[k,i]}$ near $s_1$ are $\ov K$-linear combinations of periods of $\cA_i$ near $s_k$.

Define $T$ to be a positive integer such that, for every $k = 1, \dotsc, \ell$ and $i=1, \dotsc, n$, there exists an isogeny $\cA_{[k,i],\bar\eta} \to \sigma_k^*\cA_{i,\bar\eta}$ of degree dividing~$T$.

\begin{remark} \label{no-isogenies-same-l}
Since $C$ is Hodge generic in $\AAA^n$, there are no isogenies $\cA_{i,\bar\eta} \to \cA_{i',\bar\eta}$ for distinct $i, i'$.
It follows that $(k,i) \not\sim (k,i')$ for all~$k$ and all~$i \neq i'$.
\end{remark}

The following lemma is based on the remark following Lemma~3.5 in \cite{GP23}, augmented with additional information about the degree of the isogeny.

\begin{lemma} \label{isogeny-specialisation}
Let $S$ be a connected normal noetherian scheme with geometric generic point~$\bar\eta$.
Let $A$, $B$ be abelian schemes over~$S$.
Suppose that there exists an isogeny $\phi \colon A_{\bar\eta} \to B_{\bar\eta}$.
Then, for every algebraically closed field $\ov K$ and every geometric point $\bar s \in S(\ov K)$, there exists an isogeny $\phi_{\bar s} \colon A_{\bar s} \to B_{\bar s}$ of the same degree as~$\phi$.
\end{lemma}

\begin{proof}
Let $\ell$ be a prime different from the characteristic of~$\ov K$.
Let $S[\ell^{-1}] = S \times_{\Spec(\ZZ)} \Spec(\ZZ[\ell^{-1}])$, so that $\bar s$ factors through $S[\ell^{-1}]$.
Let
\[ \rho_\ell \colon \pi_1^{et}(S[\ell^{-1}]) \to \GL(T_\ell(A \times_S B)_{\bar\eta} \otimes_{\ZZ_\ell} \QQ_\ell) \]
denote the $\ell$-adic monodromy representation, and let $G_\ell$ denote the Zariski closure of the image of~$\rho_\ell$.
Let $S'$ be the finite \'etale cover of~$S[\ell^{-1}]$ corresponding to the subgroup $\rho_\ell^{-1}(G_\ell^\circ) \subset \pi_1^{et}(S[\ell^{-1}])$.

Let $\eta'$ denote the generic point of $S'$.
By \cite[Lemma~2.8 and Lemma~3.4]{GP23}, the base-change maps
\[ \Hom_{S'}(A_{S'}, B_{S'}) \to \Hom_{\eta'}(A_{\eta'}, B_{\eta'}) \to \Hom_{\bar\eta}(A_{\bar\eta}, B_{\bar\eta}) \]
are bijective (the core of the proof of this is \cite[Prop.~I.2.7]{FC90}).
Hence $\phi \colon A_{\bar\eta} \to B_{\bar\eta}$ is the base change of a homomorphism $\phi' \colon A_{S'} \to B_{S'}$.
By \cite[Lemma~3.6]{GP23}, $\phi'$ is an isogeny.

Choose a geometric point $\bar s' \in S'(\ov K)$ which lifts~$\bar s$.
Note that $A_{\bar s}$ is equal to the specialisation of $A_{S'}$ at~$\bar s'$, and similarly for $B_{\bar s}$.
Hence, the specialisation $\phi'_{\bar s'}$ of~$\phi'$ is an isogeny $A_{\bar s} \to B_{\bar s}$.

Since $\ker(\phi')$ is a locally free finite $S'$-group scheme, $\deg(\phi'_{\bar s'}) = \deg(\phi)$.
\end{proof}

\subsection{Neighbourhoods}

It is well-known that, if $x$ is a local parameter at a point~$s_0$ on an algebraic curve~$C$ over a number field~$K$, then for every embedding $\iota \colon K \to \CC_\iota$, $x^\iota$ restricts to an analytic isomorphism between a sufficiently small neighbourhood of $s_0^\iota$ in $C^\iota(\CC_\iota)$ and an open disc in $\CC_\iota$.
The following lemma is a formal statement of the properties we shall require of these analytic isomorphisms.
A key point for us is that the discs can be taken to have radius at least~$1$ for almost all embeddings~$\iota$; since this fact is perhaps not so well-known as the rest, we have written out a proof of the lemma.

In the subsequent sections of the paper, we will only use the neighbourhood $U_{\iota,1}$.
However, the full collection of open neighbourhoods $U_{\iota,2}, \dotsc, U_{\iota,\ell}$ are needed in the proof of the properties of $U_{\iota,1}$.
In particular, (iii) is used in the proof that $x^\iota|_{U_{\iota,1}}$ is injective, and this in turn is used in the proof of~(v).


\begin{definition}
    Let $C$ be an algebraic curve over a field $K$ and let $s\in C(K)$ be a smooth $K$-point. Let $\mathfrak{m}$ denote the maximal ideal of $\cO_{C,s}$ and let $\hat{\cO}_{C,s}$ denote the $\mathfrak{m}$-adic completion of $\cO_{C,s}$. Let $x\in K(C)$ denote a local parameter for $C$ at $s$ and let $\Tay:\hat{\cO}_{C,s}\cong \powerseries{K}{X}$ denote the unique isomorphism of topological $K$-algebras which satisfies $\Tay(x)=X$. For any $f\in K(C)$ which is regular at $s$, we refer to $\Tay(f)\in \powerseries{K}{X}$ as the \defterm{Taylor series} of $f$ around $s$ in terms of $x$. 
\end{definition}

\begin{lemma} \label{padic-local-isoms}
Let $C_0$ be a smooth algebraic curve over a number field~$K$.
Let $x \in K(C_0)$ be a rational function of degree~$\ell$, such that $x$ has $\ell$ distinct unramified zeros $s_1, \dotsc, s_\ell \in C_0(K)$.
Then for each embedding $\iota$ of $K$, there exists a real number $r_\iota > 0$ and open sets $U_{\iota,1}, \dotsc, U_{\iota,\ell} \subset C_0^\iota(\CC_\iota)$ with the following properties:
\begin{enumerate}[(i)]
\item $r_\iota \geq 1$ for almost all embeddings~$\iota$ of~$K$;
\item $s_k^\iota \in U_{\iota,k}$;
\item for each $\iota$, the sets $U_{\iota,1}, \dotsc, U_{\iota,\ell}$ are pairwise disjoint and $U_{\iota,1} \cup \dotsb \cup U_{\iota,\ell}$ is equal to the preimage under $x^\iota$ of $D(0,r_\iota,\CC_\iota)$;
\item for each $\iota$ and $k$, $x^\iota$ restricts to a bijection from $U_{\iota,k}$ to the open disc $D(0,r_\iota,\CC_\iota)$;
\item for every rational function $f \in K(C_0)$ which is regular at $s_k$, if $\hat f \in \powerseries{K}{X}$ denotes the Taylor series of $f$ around $s_k$ in terms of the local parameter~$x$,
then, for all $s \in U_{\iota,k}$ satisfying $\abs{x^\iota(s)} < R(\hat f^\iota)$, we have $\hat f^\iota(x^\iota(s)) = f^\iota(s)$.
\end{enumerate}
\end{lemma}

\begin{proof}
For an archimedean embedding $\iota$, the existence of $r_\iota$ and $U_{\iota,1}, \dotsc, U_{\iota,\ell}$ satisfying (ii)--(v) is classical.
We can ignore finitely many embeddings in~(i), so for the rest of this proof we consider only non-archimedean embeddings.

Let $\ov{C_0}$ be the smooth projective model of $C_0$.
Embed $\ov{C_0}$ in some projective space $\PP^m$.
After applying a linear automorphism of $\PP^m$, we may assume that $s_1, \dotsc, s_\ell \in \AAA^m$ (regarded as an open subset of $\PP^m$).
Let $C_0' = C_0 \cap \AAA^m$.
Let $y_1, \dotsc, y_m \colon C_0' \to \AAA^1$ denote the coordinate functions.

Let $Z$ denote the finite set of $\bar K$-points of $(\ov{C_0} \cap \AAA^m) \setminus C_0'$.

Let $P, Q \in K[Y_1, \dotsc, Y_m]$ be polynomials such that
\begin{equation} \label{eqn:x-PQ}
x = \frac{P(y_1, \dotsc, y_m)}{Q(y_1, \dotsc, y_m)}.
\end{equation}
Let $\eta_{k,j} \in \powerseries{K}{X}$ denote the Taylor series of $y_j$ around $s_k$ in terms of the local parameter $x$.

\medskip

For each non-archimedean embedding~$\iota$, let
\begin{align*}
    d_\iota
  & = \min \bigl( \{ 1 \}
\\&   \phantom{= \min\bigl(} \cup \{ \abs{y_j(s_k)^\iota-y_j(s_{k'})^\iota} : 1 \leq k, k' \leq \ell, \, 1 \leq j \leq m, \, y_j(s_k) \neq y_j(s_{k'}) \}
\\&   \phantom{= \min\bigl(} \cup \{ \abs{y_j(s_k)^\iota-y_j(t)^\iota} : 1 \leq k \leq \ell, \, 1 \leq j \leq m, \, t \in Z, \, y_j(s_k) \neq y_j(t) \} \bigr),
\\  r_\iota 
  & = \min \{ R^\dag(\eta_{k,j}^\iota) d_\iota : 1 \leq k \leq \ell, \, 1 \leq j \leq m \}.
\end{align*}
For each $j,k,k'$, $y_j(s_k)-y_j(s_{k'}) \in \Qbar$, so if $y_j(s_k) \neq y_j(s_{k'})$, then $\abs{y_j(s_k)^\iota-y_j(s_{k'})^\iota} > 0$ for all~$\iota$ and $\abs{y_j(s_k)^\iota-y_j(s_{k'})^\iota} \geq 1$ for almost all~$\iota$.
A similar argument applies to $y_j(s_k)-y_j(t)$ for $t \in Z$.
Hence $d_\iota > 0$ for all $\iota$ and $d_\iota \geq 1$ for almost all~$\iota$.

The field of rational functions $K(C_0')$ has transcendence degree~$1$, so each $y_j$ is algebraic over $K(x)$.
It follows that the power series $\eta_{k,j}$ are algebraic (over $K(X)$).
Hence, by a theorem of Eisenstein \cite[\S 84]{Die57}, we have $R^\dag(\eta_{k,j}^\iota) > 0$ for all $\iota$ and $R^\dag(\eta_{k,j}^\iota) \geq 1$ for almost all~$\iota$.
Thus $r_\iota > 0$ for all $\iota$ and $r_\iota \geq 1$ for almost all~$\iota$, proving~(i).

Let
\[ U_{\iota,k} = \{ s \in (\ov C_0 \cap \AAA^m)^\iota(\CC_\iota) : \abs{x^\iota(s)} < r_\iota, \abs{y_j^\iota(s)-y_j(s_k)^\iota} < d_\iota \text{ for all } j = 1, \dotsc, m \}. \]
Clearly $s_k^\iota \in U_{\iota,k}$, so (ii) is satisfied.

For each $t \in Z$ and each~$k$, there is some $j$ such that $y_j(s_k) \neq y_j(t)$.
Consequently, $t^\iota \not\in U_{\iota,k}$.
Thus $U_{\iota,k} \subset (C_0')^\iota(\CC_\iota) \subset C_0^\iota(\CC_\iota)$.

\medskip

Next we prove (iii) and~(iv).
Let $\underline\eta_k^\iota = (\eta_{k,1}^\iota, \dotsc, \eta_{k,n}^\iota) \colon D(0,r_\iota,\CC_\iota) \to \AAA^m(\CC_\iota)$.
We shall show that $\underline\eta_k^\iota$ is the inverse of $x^\iota|_{U_{\iota,k}}$.

Firstly, suppose that $s \in U_{\iota,k} \cap U_{\iota,k'}$ where $k \neq k'$.
There exists some~$j$ such that $y_j(s_k) \neq y_j(s_{k'})$.
Then from the definitions of $U_{\iota,k}$ and $d_\iota$, we have
\begin{gather*}
   \abs{y_j^\iota(s)-y_j(s_k)^\iota} < d_\iota \leq \abs{y_j(s_k)^\iota-y_j(s_{k'})^\iota},
\\ \abs{y_j^\iota(s)-y_j(s_{k'})^\iota} < d_\iota \leq \abs{y_j(s_k)^\iota-y_j(s_{k'})^\iota}.
\end{gather*}
This contradicts the non-archimedean triangle inequality.
Thus $U_{\iota,1}, \dotsc, U_{\iota,\ell}$ are pairwise disjoint.

Consider $z \in D(0,r_\iota,\CC_\iota)$.
For every $f$ in the ideal of $\mathcal{I}(\ov C_0 \cap \AAA^m) \subset K[Y_1, \dotsc, Y_m]$, we have $f(y_1, \dotsc, y_m)=0$ in $K[C_0]$, hence $f(\eta_{k,1}, \dotsc, \eta_{k,m}) = 0$ in $\powerseries{K}{X}$, hence $f^\iota(\eta_{k,1}^\iota(z), \dotsc, \eta_{k,m}^\iota(z)) = 0$ in $\CC_\iota$ (because taking Taylor series and evaluating $\iota$-adically are both $K$-algebra homomorphisms).
Thus, $\underline\eta_k^\iota(z) \in (\ov C_0 \cap \AAA^m)^\iota(\CC_\iota)$.
A similar argument shows that
\[ X Q(\eta_{k,1}, \dotsc, \eta_{k,m}) = P(\eta_{k,1}, \dotsc, \eta_{k,m}) \text{ in } \powerseries{K}{X} \]
so
\[ z Q^\iota(\eta_{k,1}^\iota(z), \dotsc, \eta_{k,m}^\iota(z))
   = P^\iota(\eta_{k,1}^\iota(z), \dotsc, \eta_{k,m}^\iota(z)). \]
Hence,
\begin{equation} \label{eqn:x-eta}
x^\iota(\underline\eta_k^\iota(z)) = \frac{P^\iota(\underline\eta_k^\iota(z))}{Q^\iota(\underline\eta_k^\iota(z))} = z.
\end{equation}

Since $\abs{z} < r_\iota \leq R^\dag(\eta_{k,j}^\iota)$, using \cref{padic-Rdag-bound}, we have
\[ \abs{y_j^\iota(\underline\eta_k^\iota(z)) - y_j(s_k)^\iota}
   = \abs{\eta_{k,j}^\iota(z) - \eta_{k,j}^\iota(0)}
   \leq R^\dag(\eta_{k,j}^\iota)^{-1} \abs{z}
   < R^\dag(\eta_{k,j}^\iota)r_\iota
   \leq d_\iota \]
for all $j$.
Thus $\underline\eta_k^\iota(z) \in U_{\iota,k}$.

By \eqref{eqn:x-eta}, the points $\underline\eta_1^\iota(z) \in U_{\iota,1}$, \ldots, $\underline\eta_\ell^\iota(z) \in U_{\iota,\ell}$ are preimages in $C_0^\iota(\CC_\iota)$ of $z$ under $x^\iota$.
These points are distinct because $U_{\iota,1}, \dotsc, U_{\iota,\ell}$ are pairwise disjoint.
Since $\deg(x)=\ell$, these are the only preimages of $z$ under $x^\iota$ in $C_0^\iota(\CC_\iota)$.

This establishes that $U_{\iota,1} \cup \dotsb \cup U_{\iota,\ell} = (x^\iota)^{-1}(D(0,r_\iota,\CC_\iota))$ in $C_0^\iota(\CC_\iota)$, and that each $z \in D(0,r_\iota,\CC_\iota)$ has at most one preimage in each $U_{\iota,k}$ under $x^\iota$.
Thanks to \eqref{eqn:x-eta}, this implies that $x^\iota|_{U_{\iota,k}}$ is injective.
It is immediate from \eqref{eqn:x-eta} that $x^\iota|_{U_{\iota,k}}$ is surjective onto $D(0,r_\iota,\CC_\iota)$.
This completes the proof of (iii) and~(iv).

\medskip

Let $f \in K(C_0)$ be a rational function which is regular at $s_k$.
Let $\hat f \in \powerseries{K}{X}$ be the Taylor series of $f$ around $s_k$ in terms of~$x$.
Write $f = R/S$, where $R, S \in K[Y_1, \dotsc, Y_m]$ and $S^\iota$ has no zeros in $U_{\iota,k}$ except at poles of~$f^\iota$.

Suppose that $s \in U_{\iota,k}$ with $\abs{x^\iota(s)} < R(\hat f^\iota)$.
Let $z = x^\iota(s)$.
By the same argument as for \eqref{eqn:x-eta}, we obtain
\begin{equation} \label{eqn:hat-f}
\hat f^\iota(z) = \frac{R^\iota(\eta_{k,1}^\iota(z), \dotsc, \eta_{k,m}^\iota(z))}{S^\iota(\eta_{k,1}^\iota(z), \dotsc, \eta_{k,m}^\iota(z))}.
\end{equation}
Since $x^\iota$ restricts to a bijection $U_{\iota,k} \to D(0, r_\iota, \CC_\iota)$ with inverse $\underline\eta_k^\iota$, we have $\underline\eta_k^\iota(z) = s$.
Substituting this in \eqref{eqn:hat-f} completes the proof of~(v).
\end{proof}

\begin{remark} \label{move-to-U1}
We apply \cref{padic-local-isoms} to the curve~$C_0$ defined at the beginning of section~\ref{subsec:semiab}.
Then, if $s \in C_0^\iota(\CC_\iota)$ and $\abs{x^\iota(s)} < r_\iota$, then there exists $k$ such that $(\sigma_k^\iota)^{-1}(s) \in U_{\iota,1}$.
Indeed, by \cref{padic-local-isoms}(iv), there is a unique point $s' \in U_{\iota,1}$ such that $x^\iota(s') = x^\iota(s)$.
By \cref{curve-ramification}(iii), $x \colon C_4 \to \PP^1$ is Galois, so there exists $\sigma \in \Aut_x(C_4)$ such that $\sigma^\iota(s') = s$.
By \cref{sigma-ell}, $\sigma=\sigma_k$ for some~$k$.
\end{remark}

\subsection{Power series}

We now define power series $F_\lambda, G_\lambda \in \powerseries{K}{X}$ for each $\lambda \in \Lambda$.
As we will show in the next subsection, the archimedean evaluations of these power series can be interpreted as locally invariant periods of $\cA_\lambda$ near~$s_1$.
These will be the G-functions to which we apply Bombieri's theorem.

Let $\xi_\lambda$ denote the Taylor series of $\rho_\lambda \colon C_0 \to S \subset \AAA^1$ around $s_1$ in terms of the local parameter~$x$.
(Note that, since $x\circ \sigma_{k'} = x$, this is the same as the Taylor series of $\rho_i$ around $s_k$ in terms of~$x$, where $\lambda = (k,i)$.)
These are power series in $\powerseries{K}{X}$ and, since $\rho_\lambda(s_1) = \rho_i(s_k)=0$, they have no constant term.

Shrinking $r_\iota$ and $U_{\iota,1}$ if necessary, we may assume that
\[ r_\iota \leq R^\dag(\xi_\lambda^\iota) \]
for all non-archimedean $\iota$ and all $(k, i) \in \Lambda$.
Note that the power series $\xi_\lambda$ are algebraic, hence globally bounded.  Hence, we may perform this shrinking while preserving the property that
\begin{equation} \label{eqn:riota}
r_\iota > 0 \text{ for all } \iota \text{ and } r_\iota \geq 1 \text{ for almost all } \iota.
\end{equation}
Similarly, for archimedean~$\iota$, shrinking $r_\iota$ and $U_{\iota,1}$ if necessary, we may assume that
\[ r_\iota \leq R(\xi_\lambda^\iota) \]
and that $\xi_\lambda^\iota$ maps $D(0, r_\iota, \CC)$ into the disc $\Delta_S$ from Section~\ref{sec:families}.

\medskip

Let $F$ and~$G$ denote the power series defined in Section~\ref{sec:families}.
For each $\lambda \in \Lambda$, let
\begin{equation} \label{eqn:FG-li}
F_\lambda = F \circ \xi_\lambda, \; G_\lambda = G \circ \xi_\lambda \; \in \powerseries{K}{X}.
\end{equation}
Since $F \in \powerseries{\ZZ}{X}$, $R(F^\iota) \geq 1$ for all non-archimedean embeddings~$\iota$ of~$K$.
Hence, by \cref{padic-composition},
\[ R(F_\lambda^\iota) \geq R^\dag(\xi_\lambda^\iota) \geq r_\iota. \]
Using \cref{padic-composition} together with \cref{padic-local-isoms}(v), the following diagram of $\iota$-adic evaluations commutes:
\begin{equation} \label{eqn:F-xi-diagram}
\begin{tikzcd}
    {U_{\iota,1}}                  \arrow[rd, "\rho_\lambda^\iota"]
                                   \arrow[d, "x^\iota", swap]
\\  {D(0, r_\iota, \CC_\iota)}     \arrow[r, "{\xi_\lambda^\iota}"]
                                   \arrow[rr, "{F_\lambda^\iota}", bend right=20]
  & {D(0, R(F^\iota), \CC_\iota)}  \arrow[r, "F^\iota"]
  & \PP^1(\CC_\iota)
\end{tikzcd}
\end{equation}
We have not estimated $R(G^\iota)$ for non-archimedean~$\iota$, so we do not know whether this diagram makes sense with $G$ in place of~$F$.  This is not a problem because we will not need to reason about non-archimedean evaluations of $G_\lambda$.

For archimedean embeddings~$\iota$, $\xi_\lambda^\iota$ maps $D(0, r_\iota, \CC_\iota)$ to $\Delta_S$, which is contained in $D(0, R(F^\iota), \CC_\iota)$ and in $D(0, R(F^\iota), \CC_\iota)$.  Hence, all evaluations of power series in the above diagram make sense.  Over $\CC$, this is sufficient to guarantee that the diagram commutes, as well as the analogous diagram for~$G$.

\subsection{Periods and G-functions} \label{subsec:periods-g-functions}

Let $\lambda = (k,i) \in \Lambda$ and let $\iota$ be an archimedean embedding of $\hat K$.
We shall interpret the $\iota$-adic evaluation of the power series $F_\lambda$ and $G_\lambda$ as periods of the semiabelian scheme $\cA_\lambda^\iota$ near~$s_1$.
This interpretation has two applications: showing that $F_\lambda$ and $G_\lambda$ are G-functions, and determining the functional relations between them.

Write $\pi_\lambda \colon \cA_\lambda \to C_0$ for the structure morphism of the semiabelian scheme $\cA_\lambda$.
Let $\theta_\lambda$ be the morphism which completes the following fibre product diagram:
\begin{equation} \label{eqn:theta-lambda-diagram}
\begin{tikzcd}
    \cA_\lambda  \arrow[d, "\pi_\lambda"] \arrow[r, "\theta_\lambda"]
  & \cE          \arrow[d, "\pi"]
\\  C_0          \arrow[r, "\rho_\lambda|_{C_0}"]
  & S
\end{tikzcd}
\end{equation}

In Section~\ref{subsec:periods-defs}, we defined a non-zero locally invariant cycle $\gamma \in R_1 \pi^\iota_*(\bQ)(\Delta_S)$.
For each $s \in C_0^{\iota*}(\CC)$, $\theta^\iota_{\lambda,s}$ is an isomorphism $\cA^\iota_{\lambda,s} \to \cE^\iota_{\rho^\iota_\lambda(s)}$ and hence induces an isomorphism $\theta^\iota_{\lambda,s*} \colon H_1(\cA^\iota_{\lambda,s}, \QQ) \to H_1(\cE^\iota_{\rho_\lambda^\iota(s)}, \QQ)$.
Consequently, for each $s \in U_{\iota,1}$, there is a unique $\gamma_{\iota,\lambda,s} \in H_1(\cA^\iota_{\lambda,s}, \QQ)$ such that
\[ \theta^\iota_{\lambda,s*}(\gamma_{\iota,\lambda,s}) = \gamma_{\rho^\iota_\lambda(s)}. \]
These cycles $\gamma_{\iota,\lambda,s}$ are locally constant as $s$ varies in $U_{\iota,1}^*$.
Since $\gamma$ is a section in $\pi^\iota_*(\bQ)(\Delta_S^*)$, it follows that $\gamma_{\iota,\lambda,s}$ are the values of a section in $R_1 \pi^\iota_{\lambda*}(\QQ)(U_{\iota,1}^*)$.
By \cite[IX, 4.3, (4.3.2)]{And89}, this extends uniquely to a section $\gamma_{\iota,\lambda} \in R_1 \pi^\iota_{\lambda*}(\QQ)(U_{\iota,1})$.


Let $[\omega], \eta \in H^1_{DR}(\cE|_{S^*}/S^*)$ denote the de Rham cohomology classes defined in sections \ref{subsec:j-family} and~\ref{subsec:periods-defs}, respectively.
Base-changing these via the fibre product diagram~\eqref{eqn:theta-lambda-diagram}, we obtain
\[ \omega_\lambda = \rho_\lambda^* [\omega], \; \eta_\lambda = \rho_\lambda^* \eta \; \in H^1_{DR}(\cA_\lambda/C_0). \]

From \eqref{eqn:arch-F-integral} and \eqref{eqn:F-xi-diagram}, for every $s \in U_{\iota,1}$,
\begin{equation} \label{eqn:arch-Fli-integral}
   \frac{1}{2\pi i} \int_{\gamma_{\iota,\lambda,s}} \omega_{\lambda,s}^\iota
 = \frac{1}{2\pi i} \int_{\gamma_{\rho_\lambda^\iota(s)}} \omega^\iota_{\rho_\lambda^\iota(s)}
 = F^\iota(\rho_\lambda^\iota(s))
 = F_\lambda^\iota(x^\iota(s))
\end{equation}
and, similarly,
\[ \frac{1}{2\pi i} \int_{\gamma_{\iota,\lambda,s}} \eta_{\lambda,s}^\iota = G_\lambda^\iota(x^\iota(s)). \]

\begin{lemma} \label{G-functions}
$F_\lambda$ and $G_\lambda$ are G-functions for all $\lambda \in \Lambda$.
\end{lemma}

\begin{proof}
Since $\pi_\lambda|_{C_0^*}$ is an abelian scheme with multiplicative reduction at $s_1$, we can apply  \cite[X, 4.2, Theorem~1]{And89} to deduce that there is some basis $\omega', \eta'$ for $H^1_{DR}(\cA_\lambda/C_0)$ for which the Taylor expansions in $x$ of the locally invariant periods $\frac{1}{2\pi i} \int_{\gamma_{\iota,\lambda,s}} \omega_{\lambda,s}^{\prime\iota}$, $\frac{1}{2\pi i} \int_{\gamma_{\iota,\lambda,s}} \eta_{\lambda,s}^{\prime\iota}$ are G-functions.
Now $[\omega]$ and $\eta$ are $K(C_0)$-linear combinations of $\omega'$ and $\eta'$.
Hence $F_\lambda$ and $G_\lambda$ are $K(X)$-linear combinations of G-functions.
Thus, by \cite[Theorem~D]{And89}, $F_\lambda$ and $G_\lambda$ are G-functions themselves.
\end{proof}

\subsection{Global relations} \label{subsec:global-relations}

Let $\mathcal{G}$ denote the set of G-functions defined in~\eqref{eqn:FG-li}:
\begin{align*}
\mathcal{G} = \{ F_\lambda,G_\lambda : \lambda \in \Lambda \}.
\end{align*}

In order to control the radii within which global relations have to hold, we shall add an additional G-function~$H$ to our set.
Choose $\zeta\in K^\times$ such that $\abs{\zeta^\iota} \leq r_\iota$ for all embeddings $\iota$ of~$K$ at which $r_\iota < 1$.
(By \eqref{eqn:riota}, we have only constrained $\abs{\zeta^\iota}$ at finitely many embeddings, so such a $\zeta$ exists.)
Let $H$ denote the G-function $\zeta/(\zeta-X) = \sum_{k=0}^\infty(X/\zeta)^k$ and let $\mathcal{G}'=\mathcal{G}\cup\{H\}$.
This additional $G$-function~$H$ will not appear directly in our relations, but because it reduces the number of embeddings $\hat\iota$ at which a given parameter value~$z$ satisfies the radius bound, some polynomials in $\hat K[F_\lambda, G_\lambda]$ may be global relations between the elements of $\mathcal{G}'$ at a given point without being global relations between the elements of $\mathcal{G}$ at that point.

Let $R'_\iota = \min \{ 1, R(H^\iota), R(F_\lambda^\iota), R(G_\lambda^\iota) : \lambda \in \Lambda \}$.
Thus, for every embedding~$\iota$,
\[ R'_\iota \leq \min(1, R(H^\iota)) = \min(1, \abs{\zeta^\iota}) \leq r_\iota. \]

\begin{lemma} \label{global-relation-applied}
Let $M$, $N$ be positive integers.
Let $s \in C_0^*(\ov\bQ)$ be a point such that there exist $i_1, i_2, i_3, i_4 \in \{ 1, \dotsc, n \}$ such that $i_1 \neq i_2$, $i_3 \neq i_4$, $\{i_1,i_2\} \neq \{i_3,i_4\}$ and:
\begin{enumerate}[(i)]
\item there exists an isogeny $\cA_{i_1,s} \to \cA_{i_2,s}$ of degree~$M$;
\item there exists an isogeny $\cA_{i_3,s} \to \cA_{i_4,s}$ of degree~$N$.
\end{enumerate}
Then there exists a non-zero global relation between the evaluations of the G-functions $\mathcal{G}'$ at $x(s)$ of degree at most $288\ell[K(s):\QQ] + 2\ell d(MT^2)$, where $T$ is defined at the end of section~\ref{subsec:semiab}.
This relation does not involve~$H$.
\end{lemma}

\begin{proof}
The proof has two parts: first we construct a polynomial $Q$, then we prove that $Q$ is a global relation.

Let $i_1, i_2, i_3, i_4$ be the indices appearing in the statement of the lemma.
The hypotheses $i_1 \neq i_2$, $i_3 \neq i_4$, $\{i_1,i_2\} \neq \{i_3,i_4\}$ imply that either $i_1,i_2,i_3,i_4$ are distinct, or that one of $i_1,i_2$ is equal to one of $i_3,i_4$ and there are no other equalities between the indices.  In the latter case,
assume without loss of generality that $i_1=i_3$.

The strategy is as follows.
For each embedding $\hat\iota$ of $K(s)$, 
if $\abs{x^\iota(s^{\hat\iota})}$ is small, then for some $k \in \{ 1, \dotsc, \ell \}$, the point $\sigma_k^{-1}(s)^{\hat\iota}$ is close to~$s_1$.
Consequently $\hat\iota$-adic locally invariant periods of the elliptic curves $\cA^{\hat\iota}_{[k,i_j],\sigma_k^{-1}(s)}$ can be interpreted as evaluations of the G-functions $F_{[k,i_j]}$ and $G_{[k,i_j]}$ at $x(s)^{\hat\iota}$.
The isogeny $\cA_{i_1,s} \to \cA_{i_2,s}$ from hypothesis~(i) gives rise to an isogeny $\cA_{[k,i_1],\sigma_k^{-1}(s)} \to \cA_{[k,i_2],\sigma_k^{-1}(s)}$, and similarly for $i_3,i_4$ (after replacing $K(s)$ by a controlled extension $\hat K$).
Hence, via \cref{global-relation}, we get an $\hat\iota$-adic relation between the evaluations of $F_{[k,i_j]}$ and $G_{[k,i_j]}$ at $x(s)$.

In the above outline of the strategy, the index~$k$ depends on the embedding~$\hat\iota$.
Therefore, we need to apply \cref{global-relation} once for each~$k$ and multiply the resulting relations together to obtain a global relation.

\medskip

For each $k = 1, \dotsc, \ell$ and $j=1,2,3,4$, write $z_{kj} = \rho_{[k,i_j]}(\sigma_k^{-1}(s)) \in K(s)$.

By the definition of $\cA_{[k,i_1]}$ and of~$T$, there exists an isogeny $\cA_{[k,i_1],\bar\eta} \to \sigma_k^*\cA_{i_1,\bar\eta}$ of degree dividing~$T$.
By \cref{isogeny-specialisation}, this implies that there exists an isogeny $\cA_{[k,i_1],\sigma_k^{-1}(s),\ov{K}} \to (\sigma_k^* \cA_{i_1})_{\sigma_k^{-1}(s),\ov{K}}$ of the same degree.
Composing with isomorphisms coming from fibre product diagrams (in particular, $\cA_{[k,i_1]} \cong \rho_{[k,i_1]}^*\cE$), we obtain an isogeny
\[ \cE_{z_{k1},\ov{K}}
\cong \cA_{[k,i_1],\sigma_k^{-1}(s),\ov{K}}
\to (\sigma_k^* \cA_{i_1})_{\sigma_k^{-1}(s),\ov{K}}
\cong \cA_{i_1,s,\ov{K}} \]
of degree dividing~$T$.
We similarly obtain an isogeny $\cE_{z_{k2},\ov{K}} \to \cA_{i_2,s,\ov{K}}$ of degree dividing~$T$.
By hypothesis, we have an isogeny $\cA_{i_1,s} \to \cA_{i_2,s}$ of degree~$M$.
We thus obtain a chain of isogenies
\[ \cE_{z_{k1},\ov{K}} \to \cA_{i_1,s,\ov{K}} \to \cA_{i_2,s,\ov{K}} \to \cE_{z_{k2},\ov{K}} \]
whose composition $\phi_{12}$ has degree dividing $MT^2$.

Similarly, we obtain an isogeny $\phi_{34} \colon \cE_{z_{k3},\ov{K}}  \to \cE_{z_{k4},\ov{K}}$ of degree dividing~$NT^2$.

By \cite[Lemma~6.1]{MW90}, $\phi_{12}$ and $\phi_{34}$ are each defined over extensions of $K(s)$ of degree at most~$12$.
Thus we can choose an extension $\hat K/K(s)$ of degree at most $144$ such that $\phi_{12}$ and $\phi_{34}$ are both defined over $\hat K$.
For the remainder of this proof, we shall take this field $\hat K$ as our base field.

Applying \cref{global-relation} to the isogenies $\phi_{12}$ and $\phi_{34}$, we obtain non-zero homogeneous polynomials $P_{\infty,k} \in \hat K[Y_1, Z_1, \dotsc, Y_4, Z_4]$ and $P_{fin,k} \in \hat K[Y_1, Y_2]$
such that
\begin{equation} \label{eqn:relation-at-rho-s:arch}
P_{\infty,k}(F^\iota(z_{k1}^{\hat\iota}), G^\iota(z_{k1}^{\hat\iota}), \dotsc, F^\iota(z_{k4}^{\hat\iota}), G^\iota(z_{k4}^{\hat\iota})) = 0
\end{equation}
for every archimedean embedding $\hat\iota$ of $\hat K$ satisfying $z_{k1}^{\hat\iota}, \dotsc, z_{k4}^{\hat\iota} \in \Delta_S$
and
\begin{equation} \label{eqn:relation-at-rho-s:non-arch}
P_{fin,k}(F^\iota(z_{k1})^{\hat\iota}, F^\iota(z_{k2}^{\hat\iota})) = 0
\end{equation}
at every non-archimedean embedding $\hat\iota$ of $\hat K$ satisfying $\abs{z_{k1}^{\hat\iota}}, \abs{z_{k2}^{\hat\iota}} < 1$.

Define a new homogeneous polynomial $Q \in \hat K[Y_{\lambda}, Z_{\lambda} : \lambda \in \Lambda]$ by
\begin{multline*}
Q(\underline{Y}, \underline{Z})
= \prod_{k=1}^\ell P_{\infty,k}(Y_{[k,i_1]}, Z_{[k,i_1]}, Y_{[k,i_2]}, Z_{[k,i_2]}, Y_{[k,i_3]}, Z_{[k,i_3]}, Y_{[k,i_4]}, Z_{[k,i_4]})
\\ \cdot P_{fin,k}(Y_{[k,i_1]}, Y_{[k,i_2]}).
\end{multline*}

We claim that $Q \neq 0$.
Indeed, if $i_1, i_2, i_3, i_4$ are distinct, then, by \cref{no-isogenies-same-l}, for each~$k$, $[k,i_1]$, $[k,i_2]$, $[k,i_3]$ and $[k,i_4]$ are distinct, so $Q \neq 0$ is immediate from the facts that $P_{\infty,k} \neq 0$ and $P_{fin,k} \neq 0$ for all~$k$.
If $i_1=i_3$ and $i_1,i_2,i_4$ are distinct, then, similarly, for each~$k$, $[k,i_1]=[k,i_3]$ while $[k,i_1]$, $[k,i_2]$ and $[k,i_4]$ are distinct so the facts that $P_{\infty,k} \not\in \langle Y_1-Y_3, Z_1-Z_3 \rangle, P_{fin,k} \neq 0$ suffice to establish that each factor of $Q$ is non-zero.

From \cref{global-relation}, we have $\deg(P_{\infty,k}) \leq 2[\hat K:\QQ]$ and $\deg(P_{fin,k}) \leq d(MT^2)$ for all~$k$.
Hence
\[ \deg(Q) \leq 2\ell[\hat K:\QQ] + 2\ell d(MT^2) \leq 288\ell[K(s):\QQ] + 2\ell d(MT^2). \]

\medskip

It remains to prove that $Q$ is a global relation between the evaluations at~$x(s)$ of the elements of $\mathcal{G}'$.
To that end,
let $\hat{\iota}$ denote an embedding of $\hat{K}$ and let $\iota$ denote its restriction to $K$. Suppose that $\abs{x^\iota(s^{\hat{\iota}})} < R'_\iota$. In particular, $\abs{x^\iota(s^{\hat{\iota}})} < r_\iota$.

By \cref{move-to-U1}, there exists some $k \in \{ 1, \dotsc, \ell \}$ such that $\sigma_k^{-1}(s)^{\hat\iota} \in U_{\iota,1}$.
(Note that this $k$ depends on $\hat\iota$.)

From diagram~\eqref{eqn:F-xi-diagram}, noting that $x \circ \sigma_k^{-1} = x$, we have
\begin{equation} \label{eqn:chase-evaluations}
F^{\iota}(z_{k1}^{\hat\iota})
= F^\iota(\rho_{[k,i_1]}^\iota(\sigma_k^{-1}(s)^{\hat\iota}))
= F_{[k,i_1]}^\iota(x^\iota(\sigma_k^{-1}(s)^{\hat\iota})))
= F_{[k,i_1]}^\iota(x^\iota(s^{\hat{\iota}})).
\end{equation}
Similar equalities hold for $i_2$, $i_3$ and $i_4$ and, at archimedean embeddings, for $G$ in place of~$F$.

If $\hat\iota$ is archimedean, then, since $\sigma_k^{-1}(s)^{\hat\iota}$ and by our choice of $r_\iota$, $\rho_{[k,i_1]}^\iota(s^{\hat\iota}) \in \Delta_S$ and so on for $i_2$, $i_3$, $i_4$.
Hence we can substitute \eqref{eqn:chase-evaluations} (and analogous equations) into \eqref{eqn:relation-at-rho-s:arch} and obtain
\[ P_\infty(F_{[k,i_1]}^\iota(x(s)^{\hat{\iota}}), G_{[k,i_1]}^\iota(x(s)^{\hat{\iota}}), \dotsc, F_{[k,i_4]}^\iota(x(s)^{\hat{\iota}}), G_{[k,i_4]}^\iota(x(s)^{\hat{\iota}})) = 0. \]
If $\hat\iota$ is non-archimedean, then, since $\abs{x^\iota(s^{\hat{\iota}})} < r_\iota \leq R^\dag(\xi_{k,i_1}^\iota)$, by \cref{padic-Rdag-bound}, we have 
\[ \abs{\rho^\iota_{[k,i_1]}(s^{\hat\iota})} = \abs{\xi^\iota_{[k,i_1]}(x^\iota(s^{\hat\iota}))} \leq R^\dag(\xi^\iota_{k,i_1})^{-1} \abs{x^\iota(s^{\hat\iota})} < 1, \]
and similarly for~$i_2$.
Hence, we can substitute \eqref{eqn:chase-evaluations} into \eqref{eqn:relation-at-rho-s:non-arch} and obtain
\[ P_{fin}(F_{[k,i_1]}^\iota(x(s)^{\hat{\iota}}), F_{[k,i_2]}^\iota(x(s)^{\hat{\iota}})) = 0. \]
In either case, the factor of $Q$ indexed by $k$ vanishes at the $\hat\iota$-adic evaluations of $\mathcal{G}'$ at $x(s)$.
Hence, $Q$ itself vanishes at these evaluations.
\end{proof}

\subsection{Functional relations}\label{subsec:trivial}

Our aim in this section is to prove that there are no non-zero (homogeneous) functional relations between the G-functions $\{ F_{\lambda}, G_{\lambda} : \lambda \in \Lambda \} \cup \{ H \}$.
For this, it suffices to consider the evaluations of $F_\lambda$, $G_\lambda$ and~$H$ at a single archimedean embedding~$\iota$.

Let $\cA$ denote the fibre product over $C_0$ of the semiabelian schemes $\cA_\lambda$ for $\lambda \in \Lambda$.
Let $\cA^* = \cA|_{C_0^*}$ and let $\pi_\cA:\mathcal{A}^*\to C_0^*$ denote the structure morphism.
Because the elements of $\Lambda$ lie in distinct classes for the equivalence relation~$\sim$, the factors of $\cA_{\bar\eta}$ are pairwise non-isogenous elliptic curves, so the generic Mumford--Tate group of $\cA^{*\iota}$ is
\begin{align*}
    \left\{(g_\lambda)\in\gGL_2^\Lambda:\det(g_\lambda) \text{ is the same for all } \lambda \in \Lambda \right\}.
\end{align*}

Let $V$ denote a simply connected open subset of $U_{\iota,1}^* \subset C_0^{*\iota}(\CC)$.
The local system $R_1\pi_{\lambda*}^\iota(\QQ)|_V$ is trivial and has rank~$2$, so we may choose a section $\delta_{\iota,\lambda} \in R_1\pi_{\lambda*}^\iota(\QQ)(V)$ such that $\gamma_{\iota,\lambda}$ and $\delta_{\iota,\lambda}$ form an oriented basis for $R_1\pi_{\lambda*}^\iota(\QQ)|_V$.

We define holomorphic functions on $V$ by
\begin{align*}
    \cP_{\lambda,1,1}(s)=\frac{1}{2\pi i} \int_{\gamma_{\iota,\lambda,s}} \omega^\iota_{\lambda,s}, \quad
  & \cP_{\lambda,1,2}(s)=\frac{1}{2\pi i} \int_{\delta_{\iota,\lambda,s}} \omega^\iota_{\lambda,s},
\\  \cP_{\lambda,2,1}(s)=\frac{1}{2\pi i} \int_{\gamma_{\iota,\lambda,s}} \eta^\iota_{\lambda,s}, \quad
  & \cP_{\lambda,1,2}(s)=\frac{1}{2\pi i} \int_{\delta_{\iota,\lambda,s}} \eta^\iota_{\lambda,s},
\end{align*}
forming the full period matrix of the elliptic curve $\cA_{\lambda,s}^\iota$ with respect to the bases chosen above.

From \eqref{eqn:arch-Fli-integral}, for all $s \in V$, we have
\begin{equation} \label{eqn:Fli-period}
   \cP_{\lambda,1,1}(s)
 = F_\lambda^\iota(x^\iota(s)).
\end{equation}
Similarly
\[ \cP_{\lambda,2,1}(s) = G_\lambda^\iota(x^\iota(s)). \]

Putting together all of these period functions $\cP_{\lambda,\cdot,\cdot}$ yields a holomorphic function $\cP:V\to\gM_2(\CC)^\Lambda$, which gives the (block diagonal) period matrix of $\cA^\iota_s$. We denote its graph in $V\times\gM_2(\CC)^\Lambda$ by $\Gamma$. The following lemma and its proof is inspired by \cite[7.2.1]{Papas}.

\begin{lemma}\label{CZar}
Let $\Theta$ denote the subvariety of $\gM_{2,\CC}$ defined by the equation $\det=(2\pi i)^{-1}$. 
The $\CC$--Zariski closure of $\Gamma$ in $C_0^{*\iota}\times \gM_2^\Lambda$ is $C_0^{*\iota}\times\Theta^\Lambda$.
\end{lemma}

\begin{proof}
The Legendre period relation states that $\Gamma$ is contained in $C_0^{*\iota}\times\Theta^\Lambda$ (cf.\ equation~\eqref{eqn:legendre-relation}). To deduce the reverse inclusion let $s\in V$ be a Hodge generic point and let $\rho:\pi_1(C_0^{*\iota}(\CC),s)\to\gSL_{2}(\QQ)^\Lambda$ denote the monodromy representation on $R_1 \pi^\iota_{\cA*}(\QQ) = \bigoplus_{\lambda \in \Lambda} R_1 \pi^\iota_{\lambda*}(\QQ)$ at $s$ (recall that the monodromy action preserves the polarization).
Let $\gH$ denote the $\QQ$-Zariski closure of the image of $\rho$.
By Andr\'e's Normal Monodromy Theorem \cite[Theorem~1]{And92}, the neutral component $\gH^\circ$ is a normal subgroup of $\gSL_2^\Lambda$. The connected normal subgroups of $\gSL_2$ are $\{1\}$ and $\gSL_2$ itself. Since the $\rho_i$ are non-constant, $\cA_\lambda^*$ are not isotrivial and so the projections $\gH^\circ\to\gSL_2$ cannot be trivial.
It follows immediately that $\gH^\circ=\gSL_2^\Lambda$.
Since $\gH \subset \gSL_2^\Lambda$, it follows that $\gH=\gH^\circ$.
We now argue exactly as in \cite[7.2.1]{Papas}.
\end{proof}

Let $\cP' \colon V \to \AAA^{2\times \Lambda}(\CC)$ denote the composition of $\cP$ with the projection $\gM_2(\CC)^\Lambda\to\AAA^{2\times \Lambda}(\CC)$ on to the first column in each factor. The following is immediate from Lemma \ref{CZar}: 

\begin{corollary}\label{QZar}
The graph $\Gamma'$ of $\cP'$ is $\CC$-Zariski dense (and hence $\ov\QQ$-Zariski dense) in $C_0^\iota\times \AAA^{2\times m}$.
\end{corollary}

\begin{proposition}\label{prop:trivial}
There is no non-zero polynomial $P \in \Qbar[X][Y_{\lambda}, Z_{\lambda} : \lambda \in \Lambda]$ such that $P(X)(F_{\lambda}(X), G_{\lambda}(X) : \lambda \in \Lambda) = 0$ in $\powerseries{\Qbar}{X}$.
\end{proposition}

\begin{proof}
Let $C_0'$ denote the complement in $C_0$ of the poles of~$x$.

By \eqref{eqn:Fli-period}, we have the following commutative diagram:
\begin{center}
  \begin{tikzcd}
 V \arrow[r, "\cP'=(\cP_{\lambda,j,1})"] \arrow[d, "x^{\iota}"'] & \AAA^{2\times \Lambda}(\CC) \\
D(0,r_\iota,\CC) \arrow[ru, "\mathcal{G}=(g_{\lambda,j})"']              &     
\end{tikzcd}  
\end{center}
where $\lambda$ runs through~$\Lambda$, $j$ runs through $\{1,2\}$, and $g_{\lambda,1} = F^\iota_\lambda$, $g_{\lambda,2} = G^\iota_\lambda$.
Let $P$ be as in the statement of the proposition (that is, $P(X)(g_{\lambda,j} : \lambda \in \Lambda, j \in \{1,2\})=0$). Then the graph of $\mathcal{G}=(g_{\lambda,j})$ is contained in the $\CC$-points of the subvariety $V(P)$ of $\AAA^1\times\AAA^{2\times \Lambda}$ defined by $P$.
It follows that $\Gamma'$ is contained in the subvariety of $C_0' \times \AAA^{2 \times \Lambda}$ defined by $\tilde{P} = P \circ (x \times \id_{\AAA^{2 \times \Lambda}}) \in \Qbar[C_0' \times \AAA^{2 \times \Lambda}]$.
We conclude by \cref{QZar} that $\tilde{P}=0$.
Since $x$ is dominant, it follows that $P=0$.
\end{proof}

\begin{corollary} \label{cor:trivial-with-H}
There is no non-zero polynomial $P \in \Qbar[X][Y_{\lambda}, Z_{\lambda} : \lambda \in \Lambda; W]$, homogeneous in $Y_\lambda, Z_\lambda$ and~$W$, such that $P(X)(F_{\lambda}(X), G_{\lambda}(X) : \lambda \in \Lambda; H) = 0$ in $\powerseries{\Qbar}{X}$.
\end{corollary}

\begin{proof}
Suppose that such a~$P$ existed.
By substituting the rational function $H \in \Qbar(X)$ for the indeterminate~$W$, and multiplying by a non-zero element of~$\Qbar[X]$ to clear denominators, we would obtain a polynomial $P' \in \Qbar[X][Y_\lambda, Z_\lambda : \lambda \in \Lambda]$ such that $P'(X)(F_\lambda, G_\lambda : \lambda \in \Lambda) = 0$.
The homogeneity of~$P$ guarantees that $P' \neq 0$, although $P'$ might no longer be homogeneous (in $Y_\lambda$ and~$Z_\lambda$).
The existence of such~$P'$ contradicts \cref{prop:trivial}.
\end{proof}

\subsection{Proof of main theorems} \label{subsec:main-proofs}

\begin{proposition} \label{andre-bound2}
For every $\epsilon > 0$, there exist constants $\newC{andre-bound2-mult}$ and $\newC{andre-bound2-exp}$ such that the following holds:
Let $M$, $N$ be positive integers.
Let $s \in C_0^*(\ov\bQ)$ be a point such that there exist $i_1, i_2, i_3, i_4 \in \{ 1, \dotsc, n \}$ such that $i_1 \neq i_2$, $i_3 \neq i_4$, $\{i_1,i_2\} \neq \{i_3,i_4\}$, and:
\begin{enumerate}[(i)]
\item there exists an isogeny $\cA_{i_1,s} \to \cA_{i_2,s}$ of degree~$M$;
\item there exists an isogeny $\cA_{i_3,s} \to \cA_{i_4,s}$ of degree~$N$.
\end{enumerate}
Then
\[ h(x(s)) < \refC{andre-bound2-mult} \max \{ [K(s):K]^{\refC{andre-bound2-exp}}, M^\epsilon \}. \]
\end{proposition}

\begin{proof}
Let $s$ be a point satisfying the hypotheses of the proposition.
By \cref{global-relation-applied}, there is a non-zero global relation~$Q$ between the evaluations at~$x(s)$ of the elements of~$\mathcal{G}'$.
By \cref{cor:trivial-with-H}, this relation is non-trivial.
By \cref{hasse-principle-E}, we have
\[ h(x(s)) \leq \refC{bombieri-mult} \deg(Q)^{\refC{bombieri-exp}} \]
where the constants $\newC{bombieri-mult}$, $\newC{bombieri-exp}$ do not depend on $s$ or $\epsilon$.
From \cref{global-relation-applied} and \eqref{eqn:dN-bound}, we have
\[ \deg(Q) \leq 288\ell[K(s):\QQ] + 2\ell d(MT^2) \leq \refC{relation-degree-mult} \max\{ [K(s):K], M^{\epsilon/\refC{bombieri-exp}} \} \]
where the constant $\newC{relation-degree-mult}$ depends on $\epsilon$ but not on~$s$.
Putting these together proves the proposition.
\end{proof}

\begin{remark} \label{rmk:andre-bound2-effectivity}
For the purposes of effectivity, we note that \cref{andre-bound2} can be proved using \cref{hasse-principle} directly instead of \cref{hasse-principle-E}.
Indeed, the elements of~$\cG'$ satisfy a system of differential equations of the form~\eqref{eqn:diff} because, for each~$\lambda$, $F_\lambda$ and $G_\lambda$ satisfy a differential equation
\[ \frac{d}{dX}
\begin{pmatrix} F_\lambda \\ G_\lambda \end{pmatrix}
=
\Gamma_\lambda
\begin{pmatrix} F_\lambda \\ G_\lambda \end{pmatrix}
\]
for some matrix $\Gamma_\lambda \in \rM_2(\Qbar(X))$ because $\omega_\lambda$ and~$\eta_\lambda$ form a $\Qbar(S)$-basis for the $\nabla_{d/dx}$-module $H^1_{DR}(\cA_{\lambda,\Qbar(C_0)}/\Qbar(C_0))$,
and $H$ is a rational function so satisfies a homogeneous $\Qbar(X)$-linear differential equation of order~$1$.
\end{remark}

\begin{proposition} \label{galois-bound2}
There exist constants $\newC{galois-bound2-mult}$ and $\newC{galois-bound2-exp}$ such that the following holds:
Let $M$, $N$ be positive integers.
Let $s \in C_0^*(\ov\bQ)$ be a point such that there exist $i_1, i_2, i_3, i_4 \in \{ 1, \dotsc, n \}$ such that $i_1 \neq i_2$, $i_3 \neq i_4$, $\{i_1,i_2\} \neq \{i_3,i_4\}$, and:
\begin{enumerate}[(i)]
\item there exists a cyclic isogeny $\cA_{i_1,s} \to \cA_{i_2,s}$ of degree~$M$;
\item there exists a cyclic isogeny $\cA_{i_3,s} \to \cA_{i_4,s}$ of degree~$N$;
\item $\cA_{i_1,s}$ and $\cA_{i_3,s}$ do not have CM.
\end{enumerate}
Then
\[ [K(s):K] > \refC{galois-bound2-mult} \max\{M,N\}^{\refC{galois-bound2-exp}}. \]
\end{proposition}

\begin{proof}
This argument is essentially the one appearing in the proof of \cite[Lemma 4.2]{HP12}.
Throughout this proof, the notation $c_k$ will denote suitably chosen positive constants which depend on the curve~$C$ but not on $M$, $N$ or~$s$.

By swapping $i_1,i_2$ with $i_3,i_4$ if necessary, we may assume without loss of generality that $M \geq N$.

By \cite[Thm.~1.4]{GR14EC} (an effective version of the main theorem of \cite{MW90}), there exists an isogeny $\cA_{i_1,s}\to\cA_{i_2,s}$ of degree
\[ M' \leq \newC* \max\{ [K(s):K], h_F(\cA_{i_1,s}) \}^{\newC*},\]
where $h_F(\cdot)$ denotes the semistable Faltings height.
By Silverman's comparison estimate between the Faltings and modular heights of elliptic curves \cite[Prop.~2.1]{SilHEC}, we have
\[ \abs{h(1/\rho_{i_1}(s))-12h_F(\cA_{i_1,s})} \leq \newC{faltings-mult} \log\max\{ 2, h(1/\rho_{i_1}(s)) \}. \]
By the standard height machine \cite[Theorem~B.3.2]{HS00}, bearing in mind that $\operatorname{NS}(\ov{C}) \cong \ZZ$, we have
\[ \abs{h(x(s)) - \newC* h(1/\rho_{i_1}(s))} \leq \newC* h(x(s)). \]

Combining the above inequalities, we obtain
\[ M' \leq \newC* \max \{ [K(s):K], h(x(s)) \}^{\refC{isog-M-exp2}}. \]
Applying \cref{andre-bound2} with $\epsilon = 1/2\newC{isog-M-exp2}$, we get
\[ M' \leq \newC* \max \{ [K(s):K]^{\newC*}, M^{1/2} \}. \]
Since $\cA_{i_1,s}$ does not have CM, the $\ZZ$-module of isogenies $\cA_{i_1,s}\to\cA_{i_2,s}$ is generated by any cyclic isogeny. Therefore, $M\leq M'$ and the proposition follows.
\end{proof}

To deduce \cref{galois-bound}, it suffices to note that if $s \in C(\ov\QQ) \setminus \Sigma$ satisfies $\Phi_M(s_{i_1}, s_{i_2}) = \Phi_N(s_{i_3}, s_{i_4}) = 0$, then we can choose $\tilde{s} \in C_4(\ov\QQ)$ with $\nu(\tilde{s}) = s$.  After removing finitely many points from $C$, we may assume that $\tilde{s} \in C_0(\ov\QQ)$.
Now $\Phi_M(s_{i_1}, s_{i_2})$ implies that there is a cyclic isogeny $\cE_{1/s_{i_1}} \to \cE_{1/s_{i_2}}$ of degree~$M$.
We have
\[ \cA_{i_1,\tilde{s}} \cong \cE_{\rho_{i_1}(\tilde{s})} = \cE_{1/s_{i_1}} \]
and similarly for $i_2$ and for $i_3$, $i_4$.
If $\cA_{i_1,\tilde{s}}$ has CM, then so does $\cA_{i_2,\tilde{s}}$ (because they are isogenous), contradicting the assumption that $s \not\in \Sigma$.
Similarly $\cA_{i_3,\tilde{s}}$ does not have CM.
Thus $\tilde{s}$, $i_1$, $i_2$, $i_3$, $i_4$ satisfy the conditions of \cref{galois-bound2}.
Since $[\QQ(s):\QQ]/[K(\tilde{s}):K]$ is bounded below by a positive constant independent of~$s$, the bound in \cref{galois-bound2} implies the bound in \cref{galois-bound}.

\subsection{Proof of Zilber--Pink for lines in \texorpdfstring{$Y(1)^n$}{Y1n}}

The following proposition involves nothing new beyond the ideas about ``modular Mordell-Lang'' from \cite{HP12}.

\begin{proposition} \label{ZP-non-const-proj}
Let $C$ be an irreducible curve in $Y(1)^n$ which is defined over $\Qbar$.
Let $I \subset \{1,\dotsc,n\}$ denote the set of indices for which the $i$-th coordinate is non-constant on~$C$ and let $\ell = \# I$.
Let $p \colon Y(1)^n \to Y(1)^\ell$ denote the projection onto the coordinates indexed by~$I$.

Let $\Xi$ denote the set of points $s \in Y(1)^n$ which are contained in some special subvariety of $Y(1)^n$ of codimension at least~$2$, but for which $p(s)$ is not contained in a special subvariety of $Y(1)^\ell$ of codimension at least~$2$.

If $C$ is not contained in a special subvariety of $Y(1)^n$ of positive codimension, then $C \cap \Xi$ is finite.
\end{proposition}

\begin{proof}
If $s \in \Xi$ lies in a special subvariety of~$Y(1)^n$ of codimension at least~$2$, but $p(s)$ does not lie in a special subvariety of~$Y(1)^\ell$ of codimension at least~$2$, then $s$ must satisfy at least one relation of type (i) or~(ii) (as stated at the beginning of section~\ref{subsec:bounds-intro}) involving at least one index which is not in~$I$.

If $s \in C$ satisfies a relation of type~(i) (a modular relation) involving two indices $i,j \not\in I$, or if it satisfies a relation of type~(ii) (a CM coordinate) involving an index $i \not\in I$, then $C$ is contained in a special subvariety of $Y(1)^n$ of positive codimension, contradicting the hypothesis of the proposition.

Thus each point in $s \in C \cap \Xi$ must satisfy a relation of type~(i) involving one index in~$I$ and one index not in~$I$, that is, a relation of the form
\[ \Phi_M(x_i(s), u_j) = 0 \]
where $i \in I$, $j \not\in I$, and $u_j \in Y(1)^n$ denotes the constant value of the coordinate $x_j$ on~$C$.
By the argument in the first paragraph of \cite[sec.~6]{HP12}, we have
\[ [\QQ(x_i(s)):\QQ] \geq \newC{U-special-mult}(u_j) M^{1/6}. \]
Since $u_j$ (for $j \not\in I$) are constant, this implies that $[\QQ(s):\QQ] \geq \newC* M^{1/6}$ for a suitable constant.
This is the analogue of \cref{galois-bound} for points~$s \in C \cap \Xi$.
The finiteness of such points follows by the usual Pila--Zannier method, as in \cite{HP12}.
\end{proof}

To prove \cref{ZP-lines}, let $C$ be a line in~$Y(1)^n$ defined over~$\Qbar$.
As in \cref{ZP-non-const-proj}, let $p \colon Y(1)^n \to Y(1)^\ell$ denote the projection onto the coordinates which are non-constant on~$C$.
Then $p(C) \subset Y(1)^\ell$ intersects infinity.
So applying \cref{mainZP} to~$p(C)$ establishes that there are only finitely many points~$s \in C$ for which $p(s)$ is contained in a special subvariety of~$Y(1)^\ell$ of codimension at least~$2$.
Combining this with \cref{ZP-non-const-proj} proves \cref{ZP-lines}.

\bibliographystyle{amsalpha}
\bibliography{ZP}

\end{document}